\documentclass[10pt,a4paper]{article}
\usepackage{mathpazo} 
\usepackage[scaled]{helvet} 
\usepackage{courier} 
\usepackage[T1]{fontenc}
\usepackage[utf8]{inputenc}
\usepackage[font=small,labelfont={bf}]{caption}
\usepackage{epigraph}
\usepackage[explicit]{titlesec}
\usepackage{fancyhdr}
\usepackage{authblk}
\usepackage{afterpage}
\usepackage{xspace}	
\usepackage{makeidx}
\usepackage{pdflscape}
\usepackage{amsmath}
\usepackage{amssymb}
\usepackage{mathtools}
\usepackage{amsthm}
\usepackage{enumerate}
\usepackage[shortlabels]{enumitem}
\usepackage{pinlabel}
\usepackage{subfig}
\usepackage{booktabs}
\usepackage{tikz-cd}
\usepackage{tikz}
\usetikzlibrary{shapes,arrows,calc}
\usetikzlibrary{matrix}
\pgfdeclarelayer{background}
\pgfsetlayers{background,main}
\usepackage{verbatim}
\usepackage[normalem]{ulem}
\usepackage{bbm}
\providecommand{\keywords}[1]{\par\smallskip\textbf{\textbf{Keywords. }} #1}
\newtheorem{theorem}{Theorem}[section]
\newtheorem{lemma}[theorem]{Lemma}
\newtheorem{proposition}[theorem]{Proposition}

\theoremstyle{definition}
\newtheorem{definition}[theorem]{Definition}
\newtheorem{condition}[theorem]{Condition}

\newtheorem{example*}{Example}

\newtheorem*{running_example*}{Running Example}
\usepackage{xcolor}

\definecolor{thm_color_red}{RGB}{245, 188, 174}
\definecolor{thm_color}{RGB}{238, 187, 240}
\definecolor{dark_green}{RGB}{0, 123, 2}
\definecolor{very_light_blue}{RGB}{211, 255, 255}
\definecolor{dark_blue}{RGB}{0, 80, 255}
\definecolor{light_orange}{RGB}{255, 165, 0}
\definecolor{red_one}{RGB}{183, 0, 0}
\definecolor{gold}{RGB}{255, 255, 240}
\definecolor{gradient_blue_red}{RGB}{91, 40, 128}
\definecolor{turquoise}{RGB}{29, 169, 255}
\definecolor{turquoise_darker}{RGB}{0, 174, 255}
\definecolor{block_bgcolor1}{RGB}{29, 169, 255}
\definecolor{red_alert}{RGB}{222, 20, 20}

\newcommand{\dd}{ \mathrm{d}}
\newcommand{\PP}{\mathrm{P}}
\newcommand{\cA}{\mathcal{A}}

\newcommand{\cC}{\mathcal{C}}

\newcommand{\bZ}{\mathbb{Z}}

\usepackage[hidelinks]{hyperref}	
\title{Large-Deviation Principles of switching Markov processes via Hamilton-Jacobi Equations}
\author[1,2]{Mark A. Peletier}
\author[1]{Mikola C. Schlottke}
\affil[1]{Department of Mathematics and Computer Science, Eindhoven University of Technology}
\affil[2]{Institute for Complex Molecular Systems (ICMS), Eindhoven University of Technology}
\begin{document}
\maketitle
{\abstract{We prove pathwise large-deviation principles of switching Markov processes by exploiting the connection to associated Hamilton-Jacobi equations, following Jin Feng's and Thomas Kurtz's method~\cite{FengKurtz2006}. In the limit that we consider, we show how the large-deviation problem in path-space reduces to a spectral problem of finding principal eigenvalues. The large-deviation rate functions are given in action-integral form.
As an application, we demonstrate how macroscopic transport properties of stochastic models of molecular motors can be deduced from an associated principal-eigenvalue problem. The precise characterization of the macroscopic velocity in terms of principal eigenvalues confirms that breaking of detailed balance is necessary for obtaining transport. In this way, we extend and unify existing results about molecular motors and place them in the framework of stochastic processes and large-deviation theory.}}
\medskip

{\keywords{Large deviations, Hamilton-Jacobi equations, Markov processes, molecular motors, eigenvalue problems, homogenization, Feng-Kurtz method.}}
\section{Introduction}
In this paper we investigate large deviations for switching Markov processes that are motivated by stochastic models of molecular motors. Molecular motors are proteins that are capable of moving along filaments in a living cell. Molecular motors such as kinesin and dynein drag vesicles along while moving and thereby transport them within the cell. 
For more background on the phenomenon of molecular motors we refer to a number of reviews~\cite{JulicherAjdariProst1997,Howard2001,KolomeiskyFisher07,Kolomeisky13}.
\medskip

Molecular motors have a \emph{directionality}: they typically move in one direction only. A central challenge in the study of such motors is to understand the origin of this directionality, and characterize the speed of movement. In fact, mathematical models of molecular motors typically show no energetic benefit in moving in one direction or the other; the directionality arises from a non-trivial interplay between the microscopic features of such models and the dynamics of the motor. As a result, understanding how directionality arises as symmetry breaking in a-directional models is somewhat of a puzzle. 
\medskip

For certain models this puzzle has been solved, at least partially. 
Hastings, Kinderlehrer and Mcleod studied stationary solutions of certain Fokker-Planck equations and found sufficient conditions for the occurrence of transport~\cite{HastingsKinderlehrerMcleod08,HastingsKinderlehrerMcLeod2008}. Vorotnikov proved sufficient conditions for transport in deterministically switching~\cite{Vorotnikov11} and randomly switching systems~\cite{Vorotnikov14}. Perthame, Souganidis, and Mirrahimi developed a dynamic point of view on systems of molecular motors~\cite{PerthameSouganidis09a, PerthameSouganidis2009Asymmetric, Mirrahimi2013}. In particular, Mirrahimi and Souganidis prove convergence of solutions of a Fokker-Planck equation to a ballistically travelling pulse, with a velocity that is characterized by a periodic cell problem. 
\medskip

In this paper we extend the results of~\cite{Mirrahimi2013} to a much broader class of systems, make explicit the connection to stochastic processes, and  place the treatment squarely in the context of large-deviation theory. In this way we elaborate on the work by Perthame, Souganids and Mirrahimi, which appears to be inspired by large-deviation theory, as evidenced by the title of~\cite{PerthameSouganidis09a} and the use of terms such as `Hamiltonian'. 

\medskip

The larger class of stochastic processes that we consider is that of \emph{switching Markov processes in a periodic setting}. This class contains different models of molecular motors as special cases, including the \emph{continuum ratchet} and \emph{discrete stochastic} models (see~\cite{Kolomeisky13} and Section~\ref{sec:model-of-molecular-motor}, as well as~\cite{PerthameSouganidis09a,PerthameSouganidis2009Asymmetric,Mirrahimi2013,HastingsKinderlehrerMcleod08,HastingsKinderlehrerMcLeod2008}). 

\medskip

The first mathematical results of this paper (Theorems~\ref{thm:results:LDP_switching_MP} and~\ref{thm:results:action_integral_representation}; see Figure~\ref{fig:flow-diagram} below) are large-deviation theorems for such switching Markov processes. These generalize results by Kumar and Popovic~\cite{KumarPopovic2017} by focusing on pathwise large deviations, while placing more restrictive assumptions on the microscopic dynamics. Furthermore, instead of assuming the comparison principle to be satisfied as in~\cite[Lemma~1]{KumarPopovic2017}, we formulate conditions that imply the comparison principle. Faggionato and Silvestri establish large-deviation principles for fully discrete, `pseudo-one-dimensional' systems~\cite{FaggionatoSilvestri17}.

\medskip

A related line of research focuses on large-deviation principles for switching diffusions in a setting where the diffusion potentials do not have small-scale oscillations. Typical results provide large-deviation rate functionals that are simple sums of small-diffusion (`Freidlin-Wentzell') and occupation (`Donsker-Varadhan') rate functionals (see e.g.~\cite{FreidlinLee96,HeYin2014,HuangMandjesSpreij2016,BudhirajaDupuisGanguly2018,KraaijSchlottke20TR}). The rapid-scale oscillation of the potentials in this paper creates a stronger intertwining between the diffusion and switching dynamics, and consequently the rate function is not a simple sum but an expression that fully combines the dynamics of both components.

\medskip

Theorems~\ref{thm:results:LDP_switching_MP} and~\ref{thm:results:action_integral_representation} recover previous convergence results such as those of Mirrahimi and Souganidis~\cite[Th.~1.1-1.2]{Mirrahimi2013}. While the methods that Mirrahimi and Souganidis apply are inspired by large-deviation theory, they do not explicitly prove large deviation principles  but convergence statements on the level of Fokker-Planck equations. By proving large-deviation principles instead, we are able to make a clear distinction between the contributions that come from general large-deviation theory on the one hand, and the model-specific contributions on the other hand. 

For instance, our results explain from a large-deviation point-of-view why the velocity $v$ can be characterized by a cell problem that can be interpreted as defining a large-deviation Hamiltonian $\mathcal H$, through $v=\mathcal H'(0)$. 
The Hamiltonian depends on the specific model, while the relation $v=\mathcal H'(0)$ is independent of the microscopic details. 
This relation then also explains the well-known fact that detailed balance (microscopic reversibility) forces zero velocity. Indeed, we prove under general conditions (Theorem~\ref{thm:results:detailed_balance_limit_I}) that detailed balance leads to a symmetric Hamiltonian. By the characterization of the velocity as~$v=\mathcal{H}'(0)$, this means that detailed balance has to be broken in order for transport to occur.

\medskip
As another example, the numerical results of Wang, Peskin and Elston suggest that there is no transport in the limit of large reaction rates~\cite[Section~4.3, Figure~8(a)]{WangPeskinElston2003}. We also recover this result by proving that in this limit regime the Hamiltonian becomes symmetric (Theorem~\ref{thm:results:symmetry_limit_II}).



\medskip
\subsubsection*{Overview of the paper}

\medskip
In Section~\ref{sec:model-of-molecular-motor}, we illustrate the general results by means of a concrete example of a stochastic molecular-motor model. This provides a `running example' with which to interpret the general results that follow. We also outline with this example the relation  to the papers of Perthame, Souganidis and Mirrahimi.
\medskip

In Section~\ref{sec:preliminaries}, we introduce the concepts that we work with in order to rigorously formulate our results. In Section~\ref{sec:main-results} we present our main results. Figure~\ref{fig:flow-diagram} summarizes the relationships between the main theorems. Theorem~\ref{thm:results:LDP_switching_MP} provides general conditions under which the so-called spatial component of a switching Markov process satisfies a large-deviation principle. We identify the \emph{Hamiltonian}~$\mathcal{H}(p)$, a principal eigenvalue, as the central ingredient. 
Under the additional assumption that $p\mapsto \mathcal H(p)$ is convex, Theorem~\ref{thm:results:action_integral_representation} establishes an \emph{action-integral representation}.  Theorems~\ref{thm:results:LDP_switching_MP} and~\ref{thm:results:action_integral_representation} highlight the arguments that come from large-deviation theory.
\medskip

We then specialize to a concrete ratchet model of molecular motors. Theorems~\ref{thm:results:LDP_cont_MM_Limit_I} and~\ref{thm:results:LDP_cont_MM_Limit_II} establish the large-deviation theorems for two limit regimes. While Theorem~\ref{thm:results:LDP_cont_MM_Limit_I} generalizes the results in~\cite{Mirrahimi2013}, Theorem~\ref{thm:results:LDP_cont_MM_Limit_II} characterizes yet another limit regime. We include this result to illustrate how the general structure of proof remains unaffected by the choice of scaling. Finally, we show the symmetry of Hamiltonians under detailed balance (Theorem~\ref{thm:results:detailed_balance_limit_I}) and in the regime of scale separation (Theorem~\ref{thm:results:symmetry_limit_II}).
\smallskip

\tikzstyle{block} = [rectangle, draw, fill=white!20, 
    text width=10em, text centered, rounded corners, minimum height=4em]
\tikzstyle{line} = [draw, very thick, color=black!50, -latex']   
\tikzstyle{horline} = [draw, very thick, color=black!50, -latex', dotted]   
\begin{figure}[ht]
\begin{tikzpicture}[node distance = 3cm, auto]
\node [block] (LDP_SMP) {LDP for switching Markov processes (Theorem~\ref{thm:results:LDP_switching_MP})};
\node [block, below of = LDP_SMP] (AIR) {Action-Integral Representation (Theorem~\ref{thm:results:action_integral_representation})};
\node [block, below left of = AIR, node distance =4cm] (MM1) {Molecular-motor model, limit I (Theorem~\ref{thm:results:LDP_cont_MM_Limit_I})};
\node [block, below right of = AIR, node distance =4cm] (MM2) {Molecular-motor model, limit II (Theorem~\ref{thm:results:LDP_cont_MM_Limit_II})};
\node [block, below of=MM1] (DB) {Symmetry of Hamiltonian I (Theorem~\ref{thm:results:detailed_balance_limit_I})};
\node [block, below of=MM2] (ScSep) {Symmetry of Hamiltonian II (Theorem~\ref{thm:results:symmetry_limit_II})};
\node (c-LDP-AIR) at ($(LDP_SMP)!0.5!(AIR)$) {};
\node [right of = c-LDP-AIR] (ro-c-LDP-AIR) {};
\node [right of = ro-c-LDP-AIR] (convex) {Assume $\mathcal{H}(\cdot)$ convex};
\node (c-AIR-MM1) at ($(AIR)!0.5!(MM1)$) {};
\node (c-AIR-MM2) at ($(AIR)!0.5!(MM2)$) {};
\node (c-AIR-MM) at ($(c-AIR-MM1)!0.5!(c-AIR-MM2)$) {};
\node [right of = c-AIR-MM] (ro-c-AIR-MM) {};
\node [right of = ro-c-AIR-MM] (specific) {Specific models};
\node (c-MM1-DB) at ($(MM1)!0.5!(DB)$) {};
\node (c-MM2-ScSep) at ($(MM2)!0.5!(ScSep)$) {};
\node (c-MM-SpCase) at ($(c-MM1-DB)!0.5!(c-MM2-ScSep)$) {};
\node [right of = c-MM-SpCase] (ro-c-MM-SpCase) {};
\node [right of = ro-c-MM-SpCase] (special) {Special cases};
\path [line] (LDP_SMP) -- (AIR);
\path [line] (AIR) -- (MM1);
\path [line] (AIR) -- (MM2);
\path [line] (MM1) -- (DB);
\path [line] (MM2) -- (ScSep);
\path [horline] (convex) -- (c-LDP-AIR);
\path [horline,shorten >=.2cm] ([yshift=.1cm]specific.west) -- ([yshift=.1cm]c-AIR-MM1.east);
\path [horline,shorten >= .2cm] ([yshift=-.1cm]specific.west) -- ([yshift=-.1cm]c-AIR-MM2.east);
\path [horline,shorten >=.2cm] ([yshift=.1cm]special.west) -- ([yshift=.1cm]c-MM1-DB.east);
\path [horline,shorten >= .2cm] ([yshift=-.1cm]special.west) -- ([yshift=-.1cm]c-MM2-ScSep.east);
\end{tikzpicture}
\caption{Overview of the results proven in this paper. 
From top to bottom, results become less general and more specific. Arrows indicate restrictions in passing from one context to the next.}
\label{fig:flow-diagram}
\end{figure}
\section{Example---large deviations for molecular motors}
\label{sec:model-of-molecular-motor}
\subsection{Definition of the system}
In this example, we consider a two-component Markov process~$(X^n,I^n)$
with values in $\mathbb{T} \times \{1,2\}$, where~$\mathbb{T} = \mathbb{R} / \mathbb{Z}$ is the one-dimensional flat torus. 
We fix the initial condition $\left(X^n(0),I^n(0)\right) =(x_0,i_0)$ for some~$(x_0,i_0) \in \mathbb{T} \times \{1,2\}$. Let~$\psi(\cdot,1)$ and~$\psi(\cdot,2)$ be smooth functions on the torus, and we write~$\psi'(x,i)$ for the derivative of~$x\mapsto \psi(x,i)$. We call these functions \emph{potentials}. The evolution of~$(X^n, I^n)$ is characterized by the stochastic differential equation
	\begin{equation}\label{eq:intro:example_SDE}
	\dd X_t^n = -\psi'\left(nX_t^n,I_t^n\right)\,\dd t + \frac{1}{\sqrt{n}}\,\dd B_t,
	\end{equation}
	where $B_t$ is a standard Brownian motion. The process~$I^n$ is a continuous-time Markov chain on~$\{1,2\}$, which evolves with \emph{jump rates}~$r_{ij}(\cdot)$ such that
	\begin{equation}
	\mathbb{P}
	\Bigl(I^n({t+\Delta t})
	=
		j\,|\,I^n(t) =i, X^n(t) =x
		\Bigr)
		=
		n \cdot r_{ij}\left(nx\right)\Delta t
	+
		\mathcal{O}(\Delta t^2),
	\quad
		\text{as}\;\Delta t \to 0.
	\label{eq:intro:example_evolution_jump_rates}
	\end{equation}
	In summary, the \emph{spatial component}~$X^n$ is a drift-diffusion process, the \emph{configurational component}~$I^n$ is a continuous-time Markov chain on~$\{1,2\}$, and the two are coupled through their respective rates. 
	For details about the rigorous construction of such switching drift-diffusion processes, we refer to~\cite[Chapter~2]{yin2010hybrid}.
	Figure~\ref{MM:fig:evolution-and-tx-diagram} depicts a typical realization of~$(X^n, I^n)$, where the trajectory of the spatial component is lifted from the torus to~$\mathbb{R}$.
	\begin{figure}[ht]
		\labellist
		\pinlabel $t$ at 1800 90
		\pinlabel  $x$ at 760 90
		\pinlabel  $\psi(x,1)$ at -50 230
		\pinlabel  $\psi(x,2)$ at -50 500
		\pinlabel  \small $1.$ at 290 130
		\pinlabel  \small $2.$ at 310 360
		\pinlabel  \small $3.$ at 410 540
		\pinlabel  \small $4.$ at 490 520
		\pinlabel  \small $5.$ at 570 400
		\pinlabel  \small $6.$ at 570 290
		\pinlabel  \small $1.$ at 1300 80
		\pinlabel  \small $2.$ at 1190 150
		\pinlabel  \small $3.$ at 1260 200
		\pinlabel  \small $4.$ at 1100 430 
		\pinlabel  \small $5.$ at 1510 285
		\pinlabel  \small $6.$ at 1330 670
		\pinlabel {\color{red_one}{$X^{n}(t)$}} at 1000 640
		\endlabellist
		\centering
		\includegraphics[scale=.18]{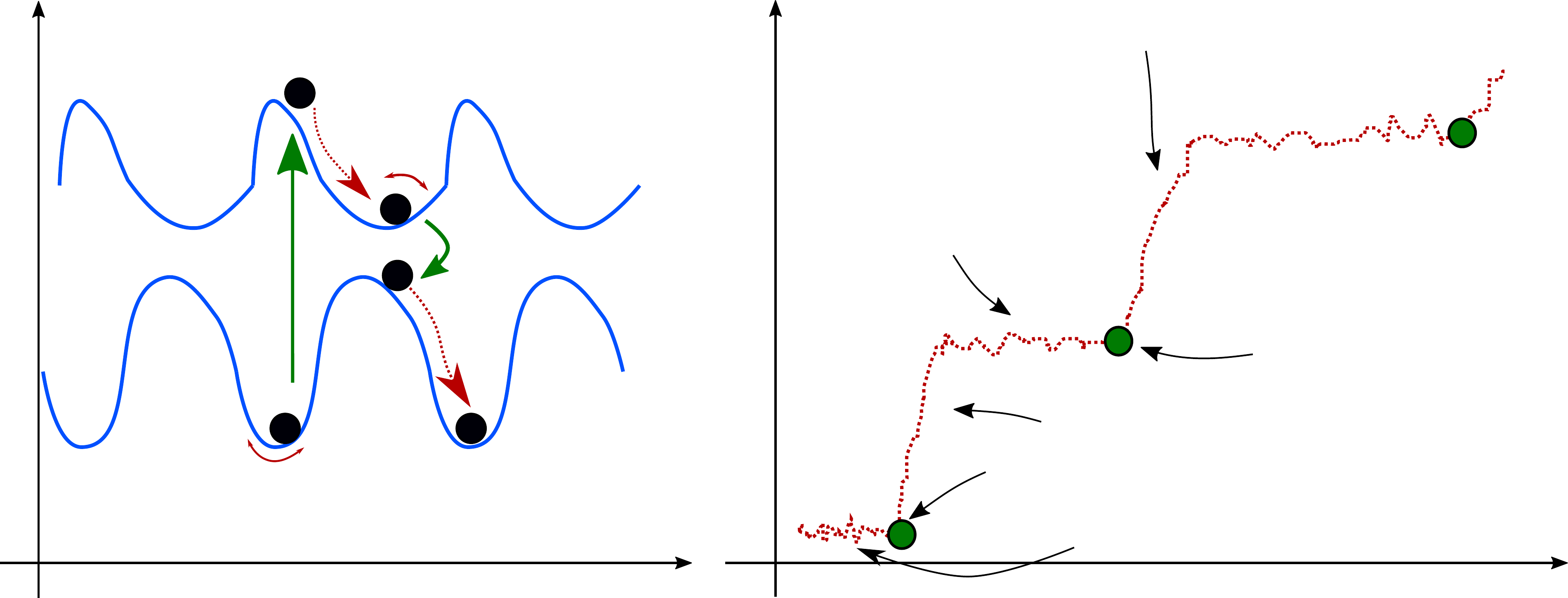}
		\caption{A typical time evolution of $(X^n, I^n)$ satisfying~\eqref{eq:intro:example_SDE} and~\eqref{eq:intro:example_evolution_jump_rates}. In the left diagram, the black bullet represents a particle that moves according to~\eqref{eq:intro:example_SDE}. A red arrow indicates the dynamics of the spatial component~$X^n$. A green arrow indicates a switch of the configurational component~$I^n$, which switches the potential in which the particle is diffusing. In the right diagram, the spatial evolution is shown in an $x$-$t$-diagram. The red dots represent the values of~$X^n$, while a green bullet indicates a switch of the configurational component~$I^n$. The dynamics of the particle comprises the following typical phases. 1 and 4:  diffusive motion of~$X^n$ near a potential minimum; 2 and 5: configurational switch of~$I^n$ with the effect of switching to another potential; 3 and 6: flow of~$X^n$ towards a minimum of the other potential. In both diagrams, the spatial trajectory is shown lifted from the torus~$\mathbb{T}$ to~$\mathbb{R}$.}
		\label{MM:fig:evolution-and-tx-diagram}
	\end{figure}
\smallskip

The specific $n$-scaling may be motivated by starting from a process~$(X_t, I_t)$ that satisfies
\begin{align*}
\dd X_t&= -\psi'(X_t,I_t)\,\dd t + \dd B_t,
\end{align*}
where the jump process $I_t$ on $\{1,2\}$ evolves according to
\begin{equation*}
\mathbb{P}\left(I_{t+\Delta t}=j\,|\,I_t=i, X_t=x\right)=r_{ij}(x)\Delta t+\mathcal{O}(\Delta t^2), \quad \text{as}\;\Delta t \to 0.
\end{equation*}
The \emph{large-scale} behaviour of $(X_t, I_t)$ is studied by considering the rescaled process~$(X^n_t, I^n_t)$ defined by~$X^n_t := \frac{1}{n} X_{n t}$ and~$I^n_t := I_{nt}$, and characterizing the dynamics of~$(X^n_t, I^n_t)$ for large values of~$n$.
This rescaling may be interpreted as zooming out of the $x$-$t$ phase space, which is illustrated below in Figure~\ref{MM:fig:zooming-out}. It\^o calculus implies that the process $(X^n_t, I^n_t)$ satisfies~\eqref{eq:intro:example_SDE} and~\eqref{eq:intro:example_evolution_jump_rates}.
\begin{figure}[ht]
	\labellist
	\pinlabel $t$ at 1000 50
	\pinlabel $x$ at 50 800
	\pinlabel $t$ at 2125 50
	\pinlabel $x$ at 1170 800
	\pinlabel $v=\mathcal{H}'(0)$ at 1850 280
	\pinlabel $n=1$ at 500 800
	\pinlabel $n\gg 1$ at 1600 800 
	\pinlabel {\color{red_one}{$X^{n}(t)$}} at 400 600
	\pinlabel {\color{red_one}{$X^n(t)$}} at 1500 600
	\endlabellist
	\centering
	\includegraphics[scale=.17]{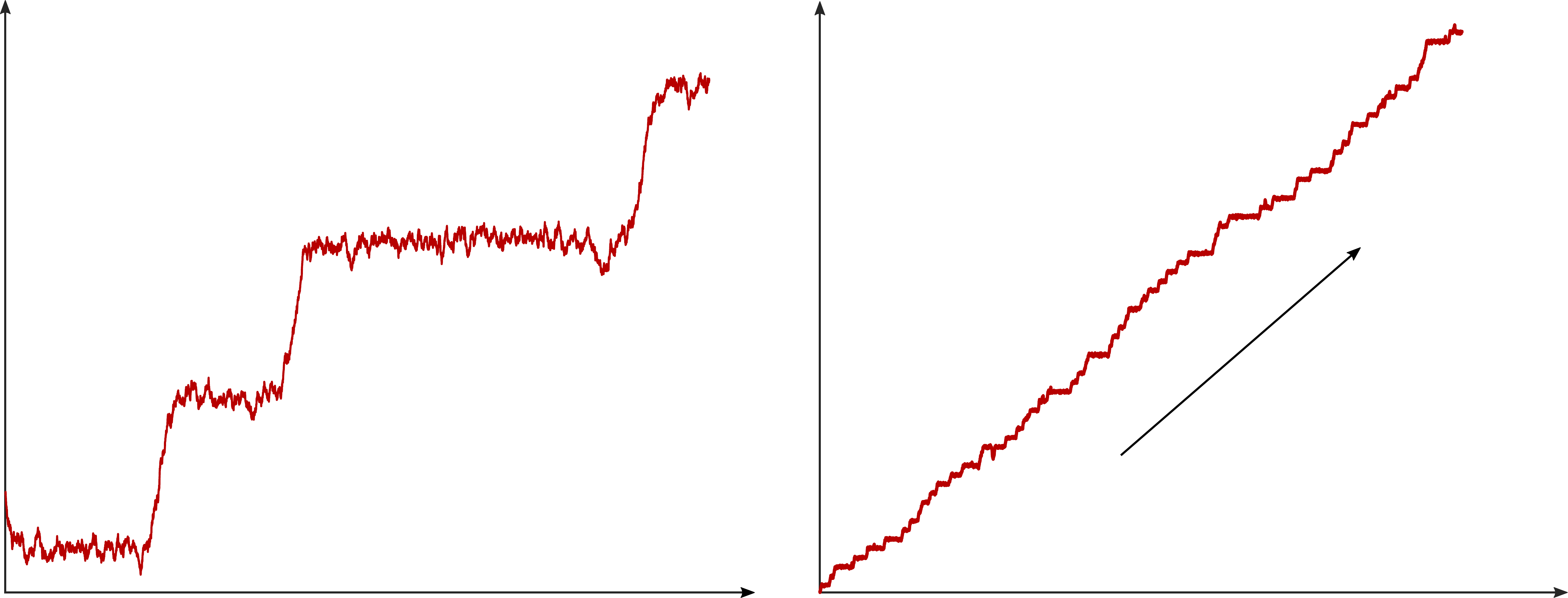}
	\caption{Two typical realizations of the spatial component~$X^n$ of the two-component process~$(X^n,I^n)$ satisfying~\eqref{eq:intro:example_SDE} and~\eqref{eq:intro:example_evolution_jump_rates}. On the left, a realization is depicted for~$n$ of order one, and on the right for large~$n$. Both graphs depict the lifted trajectory of~$X^n$ on~$\mathbb{R}$. For large~$n$, realizations of~$X^n$ closely follow a path with a constant velocity~$v=\mathcal{H}'(0)$, wherein the \emph{Hamiltonian}~$\mathcal{H}=\mathcal{H}(p)$ may be derived from large-deviation theory. A more detailed illustration of the dynamics is shown in Figure~\ref{MM:fig:evolution-and-tx-diagram} further above.}
	\label{MM:fig:zooming-out}
\end{figure}
\subsection{Large deviations for this example}
We are interested in the behaviour of the spatial component~$X^n$ as $n\to\infty$. The behaviour of~$X^n$ for large~$n$ is shown in Figure~\ref{MM:fig:zooming-out}. This figure suggests that~$X^n$ closely follows a path with a constant velocity. Indeed, when specifying the results of this paper to the example at hand---the process $(X^n, I^n)$ defined by~\eqref{eq:intro:example_SDE} and~\eqref{eq:intro:example_evolution_jump_rates}---we find that the spatial component~$X^n$ satisfies a pathwise large-deviation principle in the limit~$n\to\infty$. 
\smallskip

To describe this fact more precisely, let~$\mathcal{X}:=C_\mathbb{T}[0,\infty)$ the set of continuous trajectories in~$\mathbb{T}$, equipped with the topology of uniform convergence on compact time intervals. The spatial component~$X^n$ is a random variable in~$\mathcal{X}$, with a path distribution~$\mathbb{P}(X^n\in\cdot)\in\mathcal{P}(\mathcal{X})$. We will show that there exists a \emph{rate function}~$\mathcal{I}:\mathcal{X}\to[0,\infty]$ with which~$\{X^n\}_{n\in\mathbb{N}}$ satisfies a pathwise large-deviation principle in the sense of Definition~\ref{def:LDP} below. The gist of this statement is that for any trajectory~$x\in\mathcal{X}$, we have at least intuitively
\begin{equation}\label{MM:eq:intro:example-LDP}
\mathbb{P}\left(X^n\approx x\right) \sim e^{-n \, \mathcal{I}(x)},\quad n\to\infty.
\end{equation}
The notation~``$X^n\approx x$'' indicates that~$X^n$ is close to~$x$ with respect to the topology on~$\mathcal{X}$, and ``$\sim e^{-n \, \mathcal{I}(x)}$'' indicates a dominant contribution of the exponential.
The rate function~$\mathcal{I}$ is given by means of a \emph{Lagrangian}~$\mathcal{L}:\mathbb{R}\to[0,\infty)$ as
\begin{equation}\label{MM:eq:intro:example-LDP-RF}
\mathcal{I}(x)=\mathcal{I}_0(x(0)) + \int_0^\infty\mathcal{L}(\partial_t x(t))\,\dd t.
\end{equation}
Here~$\mathcal{I}_0:\mathbb{T}\to[0,\infty]$ is the rate function of the initial conditions~$X^n(0)$; because of the deterministic initial condition~$X^n(0)=x_0$, this functional is given by~$\mathcal{I}_0(x_0)=0$ and~$+\infty$ otherwise. The Lagrangian is the Legendre dual of a \emph{Hamiltonian}~$\mathcal{H}:\mathbb{R}\to\mathbb{R}$, that is~$\mathcal{L}(v)=\sup_p [pv-\mathcal{H}(p)]$, and the Hamiltonian is the principal eigenvalue of an associated cell problem described in a more general context in Lemma~\ref{lemma:LDP_MM:contI:principal_eigenvalue}. 
\smallskip

Here, we focus on how this large-deviation result confirms the claim suggested by Figure~\ref{MM:fig:zooming-out}. The rate function~\eqref{MM:eq:intro:example-LDP-RF} has the following properties:
	\begin{enumerate}[(i)]
	\item \label{item:i:MM-example} $\mathcal{I} : \mathcal{X} \rightarrow [0,\infty]$ is nonnegative.
	\item \label{item:ii:MM-example} $\mathcal{I}(x) = 0$ if and only if~$\partial_tx(t)=v$, with~$v=\mathcal{H}'(0)$.
	\end{enumerate}
These two properties together characterize the unique minimizer of the rate function, and thereby in particular the typical behaviour of~$X^n$ for large~$n$. Whenever~$\mathcal{I}(x) > 0$ for a path~$x\in\mathcal{X}$, then by~\eqref{MM:eq:intro:example-LDP}, the probability that a realization of~$X^n$ is close to~$x$  on~$\mathcal{X}$ is exponentially small in~$n$. In fact, the large-deviation principle implies almost-sure convergence of~$X^n$ to the unique minimizer of the rate function (Theorem~\ref{thm:LDP-implies-as}). Uniquenss of the minimizer, Item~\ref{item:ii:MM-example}, follows by strict convexity of~$\mathcal{H}(p)$. For the Hamiltonian of this example, strict convexity can be proven as demonstrated in~\cite[Step~4 in Appendix~A]{Mirrahimi2013}.
\smallskip

With the large-deviation principle we can investigate which sets of potentials and rates~$\{\psi_1,\psi_2,r_{12},r_{21}\}$ induce \emph{transport}, that means a non-zero macroscopic velocity~$v=\mathcal{H}'(0)$. We do not find general sufficient conditions for transport, but can draw some conclusions if the process~$(X^n,I^n)$ satisfies \emph{detailed balance}, that is~$r_{12}e^{-\psi_1} = C r_{21}e^{-\psi_2}$ for some constant~$C>0$. Detailed balance implies that the Hamiltonian is symmetric (Theorem~\ref{thm:results:detailed_balance_limit_I}), and therefore~$v=0$ under detailed balance.

\section{Preliminaries}
\label{sec:preliminaries}
In the previous section we sketched the  results of this paper at the hand of an example. In this section we introduce the concepts that we use in the subsequent sections to obtain the general results of this paper in a rigorous way.
\medskip

\emph{Large deviations.} For a Polish space~$E$, let~$\mathcal{X}:=D_{E}[0,\infty)$ be the set of trajectories in~$E$ that are right-continuous and have left limits. We equip~$\mathcal{X}$ with the Skorohod topology~\cite[Section~3.5]{EthierKurtz1986}. We work with the definition of a rate function as given in~\cite[Chapter~1]{BudhirajaDupuis2019}.
\begin{definition}[Rate function]
\label{def:rate-function}
	We call a map~$\mathcal{I}:\mathcal{X}\to[0,\infty]$ a \emph{rate function} if for every~$C\geq 0$, the sub-level set~$\{x\in \mathcal{X}\,:\,\mathcal{I}(x)\leq C\}$ is compact.
	\qed
\end{definition}
In particular, a rate function is lower semi-continuous. For a Borel subset~$A\subseteq \mathcal{X}$, we write~$\mathrm{int}(A)$ and~$\mathrm{clos}(A)$ for its interior and closure.
\begin{definition}[Large-deviation principle]
\label{def:LDP}
	For~$n=1,2,\dots$, let~$\PP_n$ be a probability measure on~$\mathcal{X}$, and let~$\mathcal{I}:\mathcal{X}\to[0,\infty]$ be a rate function. We say that the sequence~$\{\PP_n\}_{n\in\mathbb{N}}$ satisfies a \emph{large-deviation principle with rate function~$\mathcal{I}$} if for every Borel subset~$A\subseteq \mathcal{X}$, 
	\begin{align*}
	-\inf_{x\in\mathrm{int}(A)} \mathcal{I}(x) \leq \liminf_{n\to\infty}\frac{1}{n}\log\PP_n(A) \leq \limsup_{n\to\infty}\frac{1}{n}\log\PP_n(A) \leq -\inf_{x\in\mathrm{clos}(A)}\mathcal{I}(x)
	\tag*\qed
	\end{align*}
\end{definition}
A large-deviation principle provides an estimate of the probabilities~$\PP_n(A)$ on the logarithmic scale. At least intuitively,
\begin{equation*}
\PP_n(A) \approx e^{-n\, \inf_{x\in A}\mathcal{I}(x)},\qquad n\to \infty.
\end{equation*}
Illustrating examples of a large-deviation principle can be found for instance in Ellis' note on Boltzmann's discoveries~\cite{Ellis1999}. General introductions to the topic are also provided in~\cite[Chapter~1]{BudhirajaDupuis2019} and~\cite[Chapter~3]{FengKurtz2006}.

\medskip
Identifying tractable formulas for a rate function is crucial for drawing conclusions from a large-deviation principle. In this paper, we shall aim for finding action-integral representations of rate functions. Let~$\mathbb{T}^d := \mathbb{R}^d / \mathbb{Z}^d$ be the flat~$d$-dimensional torus, and let~$\cA\cC([0,\infty);\mathbb{T}^d)$ be the set of absolutely continuous trajectories in$~\mathbb{T}^d$.
\begin{definition}[Action-integral form of rate function]
\label{def:action-integral-rep-of-rate-function}
We say that a rate function~$\mathcal{I}:D_{\mathbb{T}^d}[0,\infty)\to[0,\infty]$ is of action-integral form if there is a non-trivial convex map~$\mathcal{L}:\mathbb{R}^d\to[0,\infty]$ with which
\begin{equation*}
\mathcal{I}(x) =
\begin{cases}
\mathcal{I}_0(x(0))+
\int_0^\infty\mathcal{L}
\left(\partial_t x(t)\right) \, \dd t 
&\quad 
\text{if } x\in \cA\cC([0,\infty); \mathbb{T}^d),\\
+\infty & 
\quad
\text{otherwise},
\end{cases}
\end{equation*}
where~$\mathcal{I}_0:\mathbb{T}^d\to[0,\infty]$ is a rate function. We refer to the map~$\mathcal{L}$ as the \emph{Lagrangian}.\qed
\end{definition}

\medskip
\noindent
\emph{Switching Markov processes in a periodic setting.}
We shall consider Markov processes defined by two-component stochastic processes~$(X^n, I^n)$ taking values in state spaces~$E_n$ that satisfy the following condition. 
\begin{condition}[Setting]
	Fix~$J\in\mathbb{N}$. For~$n\in\mathbb{N}$, the state space~$E_n$ is a product space~$E_n := E_n^X \times \{1, \dots, J\}$, where~$E_n^X$ be a compact Polish space satisfying the following: there are continuous maps~$\iota_n : E_n^X \rightarrow \mathbb{T}^d$ such that for all~$x \in \mathbb{T}^d$ there exists~$x_n \in E_n^X$ with which~$\iota_n(x_n) \rightarrow x$ as~$n \rightarrow \infty$.\qed
	\label{condition:results:general_setting}
\end{condition}
This condition means that the~$E_n^X$ are asymptotically dense in the torus~$\mathbb{T}^d$. The typical example is the periodic lattice $(n^{-1}\bZ)^d/\bZ^d$, where the torus is recovered in the limit of~$n$ to infinity. Another example is simply~$E_n^X\equiv \mathbb{T}^d$. When it is clear from the context, we omit $\iota_n$ in the notation.

\medskip
Let~$\mathcal{X}_n:=D_{E_n}[0,\infty)$. For a distribution~$\mu\in\mathcal{P}(E_n)$, we define an~$E_n$-valued two-component process~$(X^n, I^n)$ with initial condition~$\mu$ by defining its path distribution~$\mathbb{P}_\mu^n\in\mathcal{P}(\mathcal{X}_n)$. In order to define a path distribution, we shall specify a linear map~$L_n:\mathcal{D}(L_n) \subseteq C(E_n) \rightarrow C(E_n)$ on a domain~$\mathcal{D}(L_n)$ and assume well-posedness of the martingale problem of the pair~$(L_n,\mu)$; we refer to~\cite[Section~4.3]{EthierKurtz1986} for a precise treatment of the martingale problem. We call a linear map~$L_n$ as above a \emph{generator} if it gives rise to a well-posed martingale problem.
We specify the generators of~$(X^n,I^n)$ from the following ingredients:
\begin{enumerate}[(1)]
	\item For~$i\in\{1,\dots,J\}$, we have a map~$L_n^{i}:\mathcal{D}(L_n^{i}) \subseteq C(E_n^X) \rightarrow C(E_n^X)$ that is the generator of an~$E_n^X$-valued Markov process. 
	\item For~$i,j\in\{1,\dots,J\}$, we have a continuous map~$r_{ij}^n:E_n^X\to[0,\infty)$.
\end{enumerate}
With that, define the map~$L_n:\mathcal{D}(L_n) \subseteq C(E_n) \rightarrow C(E_n)$ by
\begin{equation}
L_n f(x,i) := L_n^{i} f(\cdot,i) (x) + \sum_{j = 1}^J r_{ij}^n(x) \left[ f(x,j) - f(x,i) \right],
\label{eq:intro:MP_with_switching_L_varepsilon}
\end{equation}
where the domain is~$\mathcal{D}(L_n) = \{ f \in C(E_n) \, : \, f(\cdot,i) \in \mathcal{D}(L_n^{i}), i = 1, \dots, J \}$.
\begin{condition}[Well-posedness]
	Let~$\mu \in \mathcal{P}(E_n)$. Existence and uniqueness holds for the $\mathcal{X}_n$ martingale problem for $(L_n,\mu)$. Denote the solution to the martingale-problem solution of $L_n$ by~$\mathbb{P}_\mu^n$. The map~$E_n \ni z \mapsto \mathbb{P}_{\delta_z}^n \in \mathcal{P}(\mathcal{X}_n)$ is Borel measurable with respect to the weak topology on $\mathcal{P}(\mathcal{X}_n)$.\qed
	\label{condition:intro:sol_martinagle_problem}
\end{condition}
Condition~\ref{condition:intro:sol_martinagle_problem} is the basic assumption on the processes in~\cite{FengKurtz2006}. A sufficient condition for the measurability in there is given in~\cite[Theorem~4.4.6]{EthierKurtz1986}. We do not give general conditions on a map~$L_n$ that imply Condition~\ref{condition:intro:sol_martinagle_problem}. For examples regarding existence and regularity properties we refer to the book of Yin and Zhu about switching hybrid diffusions~\cite[Part~I]{yin2010hybrid}.
\begin{definition}[Switching Markov processes in a periodic setting]
	Let~$(X^n,I^n)$ be a two-component Markov proces taking values in~$E_n = E^X_n \times \{1,\dots,J\}$ satisfying Condition~\ref{condition:results:general_setting}. We call~$(X^n,I^n)$ a \emph{switching Markov process} if its generator~$L_n$ is given by~\eqref{eq:intro:MP_with_switching_L_varepsilon} and satisfies Condition~\ref{condition:intro:sol_martinagle_problem}.\qed
	\label{def:intro:MP_with_switching}
\end{definition}
\section{Main results}
\label{sec:main-results}
In the previous section we introduced the notion of a large-deviation principle and defined switching Markov processes in a periodic setting. In this section we present our main results as depicted in the flow-diagram Figure~\ref{fig:flow-diagram} above. First, we formulate general conditions for a large-deviation principle of switching Markov processes (Theorem~\ref{thm:results:LDP_switching_MP}). Then we find an action-integral representation of the rate function under an additional convexity assumption (Theorem~\ref{thm:results:action_integral_representation}). The remaining theorems arise from specifications of the general setting to specific models. We prove large-deviation principles for molecular-motor models in two limit regimes (Theorems~\ref{thm:results:LDP_cont_MM_Limit_I} and~\ref{thm:results:LDP_cont_MM_Limit_II}), and derive the fact that detailed balance and separation of scales imply symmetry of Hamiltonians (Theorems~\ref{thm:results:detailed_balance_limit_I} and~\ref{thm:results:symmetry_limit_II}).
\subsection{Large-deviation principle for switching Markov processes}  
We consider switching Markov processes~$(X^n,I^n)$ in a periodic setting in the sense of Definition~\ref{def:intro:MP_with_switching}, with generators of the form~\eqref{eq:intro:MP_with_switching_L_varepsilon}. The essence of this section is Theorem~\ref{thm:results:LDP_switching_MP}, which provides general conditions that imply a pathwise large-deviation principle of the spatial component~$X^n$. We state the conditions in terms of \emph{nonlinear generators} defined as follows.
\begin{definition}[Nonlinear generators]
	\label{MM:def:nonlinear-generators-switching-MP}
	Let~$L_n$ be the map defined by~\eqref{eq:intro:MP_with_switching_L_varepsilon}. The~\emph{nonlinear generator} is the map~$H_n : \mathcal{D}(H_n) \subseteq C(E_n) \rightarrow C(E_n)$ defined by
	\begin{equation}
	H_n f(x) := \frac{1}{n} \, e^{-n f(x)} L_n (e^{n f(\cdot)})(x),
	\label{eq:results:H_varepsilon}
	\end{equation}
	on the domain~$\mathcal{D}(H_n) := \{f \in C(E_n) \, : \, e^{n f(\cdot)} \in \mathcal{D}(L_n)\}$.\qed
\end{definition}
We shall work under the assumption that the nonlinear generators~$H_n$ converge in the limit~$n\to \infty$. To formulate this convergence assumption, we need to introduce an additional state space~$E'$ for collecting up-scaled variables. The following diagram depicts the relation between the state spaces:
\begin{equation*}
\begin{tikzcd}[row sep = tiny]
&	\mathbb{T}^d \times E^\prime \arrow[dd, "\mathrm{proj}_1"]	\\
E_n \arrow[ur, "(\eta_n {,} \eta_n^\prime)"] \arrow[dr, "\eta_n"']	&	\\
&	\mathbb{T}^d	
\end{tikzcd}
\end{equation*}
In the diagram,~$\eta_n : E_n \to \mathbb{T}^d$ is the projection defined by~$\eta_n (x,i) := \iota_n(x)$, where~$\iota_n:E_n^X\to\mathbb{T}^d$ is the embedding of Condition~\ref{condition:results:general_setting}. The map~$\eta_n^\prime : E_n \to E^\prime$ is assumed to be continuous. We shall assume that the~$E_n$ are asymptotically dense:
\begin{enumerate}[(C1)]
	\item \label{MM:item:C1} For~$(x,z') \in \mathbb{T}^d \times E'$ there exists $y_n\in E_n$ such that $\eta_n(y_n) \to x$ and $\eta_n'(y_n) \to z^\prime$ as $n \to \infty$.
\end{enumerate}
A limit operator of~$H_n$ is defined by a graph~$H\subseteq C(\mathbb{T}^d)\times C(\mathbb{T}^d\times E')$, a multi-valued operator. We shall assume the following convergence condition:
\begin{enumerate}[(C2)]
	\item \label{MM:item:C2} The domain~$\mathcal{D}(H)$ satisfies~$C^\infty(\mathbb{T}^d) \subseteq \mathcal{D}(H) \subseteq C^1(\mathbb{T}^d)$. For~$(f,g)\in H$, there exist functions~$f_n\in\mathcal{D}(H_n)$,~$n\in\mathbb{N}$, such that as~$n\to\infty$,
	\[
	\| f\circ\eta_n - f_n \|_{L^\infty(E_n)} \rightarrow 0
	\quad
	\text{ and }
	\quad
	\| g\circ (\eta_n,\eta_n') - H_n f_n \|_{L^\infty(E_n)} \rightarrow 0.
	\]
\end{enumerate}
Frequently, for any~$f$ in the domain of~$H$, the corresponding image functions~$g$ are naturally parametrized by a set of functions on~$E'$: 
\begin{enumerate}[(C3)]
	\item \label{MM:item:C3} There are a set~$\mathcal{C}\subseteq C(E')$ and functions~$H_{f,\varphi}\in C(\mathbb{T}^d\times E')$ with which
	\begin{equation*}
	H = \left\{\left(f,H_{f,\varphi}\right)\,:\,f\in\mathcal{D}(H),\varphi\in\mathcal{C}\right\}.
	\end{equation*}
\end{enumerate}
\begin{theorem}[Large deviation principle for switching processes]
	\label{thm:results:LDP_switching_MP}
	Let~$(X^n,I^n)$ be a switching Markov process in the sense of Definition~\ref{def:intro:MP_with_switching}, with nonlinear generators~$H_n$ of Definition~\ref{MM:def:nonlinear-generators-switching-MP}. Let~$E^\prime$ be a compact metric space satisfying~\ref{MM:item:C1}, and let~$H\subseteq C(\mathbb{T}^d)\times C(\mathbb{T}^d\times E')$ be a multivalued operator satisfying ~\ref{MM:item:C2} and~\ref{MM:item:C3} from above. Suppose the following: 
	\begin{enumerate}[(T1)]
		\item \label{MM:item:T1}
		For every $\varphi\in\mathcal{C}$ there is a map~$H_\varphi : \mathbb{R}^d \times E^\prime \to \mathbb{R}$
		such that for all~$f \in \mathcal{D}(H)$,
		\begin{equation*}
		H_{f,\varphi}(x,z^\prime) = H_\varphi(\nabla f(x),z^\prime),\qquad (x,z') \in \mathbb{T}^d\times E'.
		\end{equation*}
		\item \label{MM:item:T2}
		For every $p \in \mathbb{R}^d$, there exists a function $\varphi_p \in \mathcal{C}$ and a constant $\mathcal{H}(p) \in \mathbb{R}$ such that~$H_{\varphi_p}(p,z^\prime)=\mathcal{H}(p)$
		for all $z^\prime \in E^\prime$. 
	\end{enumerate}
	Suppose furthermore that~$\{X^n(0)\}_{n\in\mathbb{N}}$ satisfies a large-deviation principle in~$\mathbb{T}^d$ with rate function~$\mathcal{I}_0 : \mathbb{T}^d \rightarrow [0,\infty]$.
	Then the family of processes $\{X^n\}_{n\in\mathbb{N}}$ satisfies a large-deviation principle in $D_{\mathbb{T}^d}[0,\infty)$ with a rate function $\mathcal{I} :D_{\mathbb{T}^d}[0,\infty)\rightarrow[0,\infty]$, and there exists a semigroup~$V(t)$ with which the rate function is given by~\eqref{eq:rate-function:FengKurtz}.
\end{theorem}
We give the proof in Section~\ref{sec:proof-of-LDP}.
The formula for the rate function~$\mathcal{I}$ is not important here, which is why we report it only  below in~\eqref{eq:rate-function:FengKurtz} in the proof section. Condition~\ref{MM:item:T1} means that the images depend on the variable~$x\in\mathbb{T}^d$ only via the gradients~$\nabla f(x)$. In the molecular-motor models, Condition~\ref{MM:item:T2} is verified by solving a principal-eigenvalue problem, in which the constant~$\mathcal{H}(p)$ is the unique principal eigenvalue of a certain cell problem. 
\subsection{Action-integral representation of the rate function}
In the previous section, we formulated general conditions that imply a pathwise large-deviation principle. The rate function of Theorem~\ref{thm:results:LDP_switching_MP} however is still generic (equation~\eqref{eq:rate-function:FengKurtz} below). The following Theorem shows that under an additional convexity assumption, the rate function is of action-integral form in the sense of Definition~\ref{def:action-integral-rep-of-rate-function} above.
\begin{theorem} 
\label{thm:results:action_integral_representation}
Consider the setting of Theorem~\ref{thm:results:LDP_switching_MP}. For~$p\in\mathbb{R}^d$, let~$\mathcal{H}(p)$ be the constant in~\ref{MM:item:T2} of Theorem \ref{thm:results:LDP_switching_MP}.
	Suppose further the following:
	\begin{itemize}
		\item[(T3)] The map $p \mapsto \mathcal{H}(p)$ is convex and $\mathcal{H}(0) = 0$.
	\end{itemize}
	Then the rate function of Theorem~\ref{thm:results:LDP_switching_MP} is of action-integral form with the Lagrangian defined by~$\mathcal{L}(v)=\sup_{p\in\mathbb{R}^d}\left[p\cdot v - \mathcal{H}(p)\right]$.
\end{theorem}
Theorem~\ref{thm:results:action_integral_representation} is proven in Section~\ref{section:action_integral}.
\subsection{Large deviations for models of molecular motors}
In the previous two sections we considered general switching Markov processes in a periodic setting. In this section we further specify to a class of stochastic processes motivated by molecular motors. 
\begin{definition}[Process modeling molecular motors]
	The pair~$(X^n, I^n)$ is a Markov process with values in~$E_n = \mathbb{T}^d \times \{1,\dots,J\}$ with generator~$L_n$ acting on functions~$f=f(x,i)$ as
	\begin{multline}
	L_n f(x,i) := b^i(nx) \cdot \nabla_x f(\cdot,i) (x) + \frac{1}{n} \frac{1}{2} \Delta_x f(\cdot,i) (x)\\
	+ \sum_{j\neq i} \gamma(n) r_{ij}(nx)
	\left[
	f(x,j) - f(x,i)
	\right],
	\label{eq:intro:L_varepsilon_cont_MM}
	\end{multline}
	where~$\gamma(n)>0$,~$r_{ij}(\cdot) \in C^\infty(\mathbb{T}^d; [0,\infty))$, and~$b^i(\cdot) \in C^\infty(\mathbb{T}^d)$. \qed
	\label{def:intro:cont_MM}
\end{definition}
This is an example of a switching Markov process with generators~$L_n^{i}$ defined on the core~$C^2(\mathbb{T}^d)$ by
\begin{equation*}
L^{i}_n g(x)
:=
b^i(nx) \cdot \nabla g(x)
+
\frac{1}{n} \,\frac{1}{2} \Delta g(x),
\end{equation*}
and rates~$r_{ij}^n(x)= \gamma(n)r_{ij}(nx)$. The domain~$\mathcal{D}(L_n^i)$ of the generators~$L_n^i$ contains the core, but is larger than~$C^2(\mathbb{T}^d)$.
The domain of~$L_n$ is the set given by~$\mathcal{D}(L_n) = \{ f(x,i)\,:\, f(\cdot,i) \in \mathcal{D}(L_n^{i})\}$, and for functions~$f$ such that~$f(\cdot,i)\in C^2(\mathbb{T}^d)$, the generator acts as defined in~\eqref{eq:intro:L_varepsilon_cont_MM}. The example of Section~\ref{sec:model-of-molecular-motor}, a stochastic model of molecular motors, corresponds to the choices~$d=1$, $b^i=-\psi'(\cdot,i)$, $J=2$ and~$\gamma(n)=n$.
\begin{definition}\label{MM:def:irreducible-matrix}
	Let~$J\in\mathbb{N}$. We call a matrix~$A\in\mathbb{R}^{J\times J}$ \emph{irreducible} if there is no decomposition of~$\{1,\dots,J\}$ into two disjoint sets $\mathcal{J}_1$ and $\mathcal{J}_2$ such that $A_{ij} = 0$ whenever $i \in \mathcal{J}_1$ and $j \in \mathcal{J}_2$.\qed
\end{definition}
\begin{theorem}[Limit I]
\label{thm:results:LDP_cont_MM_Limit_I}
	Let~$(X^n_t, I^n_t)$ be the Markov process of Definition~\ref{def:intro:cont_MM} with parameter~$\gamma(n)=n$. Assume that the matrix~$R$ with entries $R_{ij} = \sup_{y \in \mathbb{T}^d}r_{ij}(y)$ is irreducible. Suppose furthermore that
	the family of initial conditions~$\{X^n(0)\}_{n\in\mathbb{N}}$ satisfies a large-deviation principle in $\mathbb{T}^d$ with rate function $\mathcal{I}_0 :\mathbb{T}^d \rightarrow [0,\infty]$.
	\medskip
	
	Then the family of stochastic processes~$\{X^n\}_{n\in\mathbb{N}}$ satisfies a large-deviation principle in $C_{\mathbb{T}^d}[0,\infty)$ with rate function of action-integral form.
	The Hamiltonian~$\mathcal{H}(p)$ is the principal eigenvalue of an associated cell problem described in Lemma~\ref{lemma:LDP_MM:contI:principal_eigenvalue}.
\end{theorem}
The irreducibility condition is imposed to solve the principal-eigenvalue problem that we obtain, and is inspired by sufficient conditions for solvability of a coupled system of elliptic PDEs~\cite{Sweers92}. 
\medskip

The parameter~$\gamma(n)$ allows to model a time-scale separation of the components. The following theorem shows that if~$\gamma(n)$ scales super-linearly, then the spatial component is effectively driven by potentials averaged over the stationary measure of the fast configurational component, and the large-deviation principle is governed by an averaged Hamiltonian.
\begin{theorem}[Limit II]
	Let $(X^n_t, I^n_t)$ be the Markov process of Definition~\ref{def:intro:cont_MM}, with parameter~$\gamma(n)$ such that~$n^{-1}\gamma(n) \rightarrow \infty$ as~$n \rightarrow \infty$. Assume that for every~$y\in\mathbb{T}^d$, the matrix~$R(y)$ with entries~$R(y)_{ij}=r_{ij}(y)$ is irreducible.
	Suppose furthermore that the family of random variables~$\{X^n(0)\}_{n \in \mathbb{N}}$ satisfies a large-deviation principle in~$\mathbb{T}^d$ with rate function~$\mathcal{I}_0 :\mathbb{T}^d \rightarrow [0,\infty]$. 
	\smallskip
	
	Then~$\{X^n\}_{n \in \mathbb{N}}$ satisfies a large-deviation principle in~$C_{\mathbb{T}^d}[0,\infty)$ with rate function of action-integral form.
	The Hamiltonian $\overline{\mathcal{H}}(p)$ is the principal eigenvalue of an associated averaged cell problem described in Lemma~\ref{lemma:LDP_MM:contII:principal_eigenvalue}.
	\label{thm:results:LDP_cont_MM_Limit_II}
\end{theorem}
\subsection{Detailed balance implies symmetric Hamiltonians}
The large-deviation principles established by Theorems~\ref{thm:results:LDP_cont_MM_Limit_I} and~\ref{thm:results:LDP_cont_MM_Limit_II} can be used to analyse which sets of potentials and rates induce transport on macroscopic scales. To that end, we specify to~$b^i(y)=-\nabla_y\psi_i(y)$ and~$\gamma(n)=n$ in the generators defined in~\eqref{eq:intro:L_varepsilon_cont_MM}. We say that the set of potentials and rates~$\{r_{ij},\psi_i\}$ satisfies \emph{detailed balance} if for all~$i,j\in\{1,\dots,J\}$ and~$y\in\mathbb{T}^d$, we have
\begin{align}\label{MM:eq:detailed-balance}
r_{ij}(y) e^{-2\psi_i(y)}=r_{ji}(y) e^{-2\psi_j(y)}.
\end{align}
\begin{theorem}[Detailed balance implies a symmetric Hamiltonian]
\label{thm:results:detailed_balance_limit_I}
Consider the same setting and assumptions of Theorem~\ref{thm:results:LDP_cont_MM_Limit_I}. Suppose that the detailed-balance condition~\eqref{MM:eq:detailed-balance} is satisfied. Then the Hamiltonian~$\mathcal{H}(p)$ of Theorem~\ref{thm:results:LDP_cont_MM_Limit_I} satisfies~$\mathcal{H}(p) = \mathcal{H}(-p)$ for all~$p\in\mathbb{R}^d$.
\end{theorem}
We give the proof of Theorem~\ref{thm:results:detailed_balance_limit_I} here, since it is solely based on a suitable formula for~$\mathcal{H}(p)$.
\begin{proof}[Proof of Theorem~\ref{thm:results:detailed_balance_limit_I}]
	We prove in Proposition~\ref{prop:results:detailed_balance_limit_I} that under the detailed-balance condition, the principal eigenvalue~$\mathcal{H}(p)$ is given by 
	$$
	\mathcal{H}(p)
	=
	\sup_{\mu \in \mathbf{P}} \left[ K_p(\mu) - \mathcal{R}(\mu) \right],
	$$
	where $\mathbf{P} \subset \mathcal{P}(E^\prime)$ is a subset of probability measures on $E^\prime = \mathbb{T}^d \times \{1, \dots, J\}$ specified in Proposition~\ref{prop:results:detailed_balance_limit_I}, $\mathcal{R}(\mu)$ is the relative Fisher information specified in~\eqref{MM:eq:relative-Fisher-information}, and $K_p(\mu)$ is given by
	\begin{multline*}
	K_p(\mu)= 
	\inf_{\phi}\bigg\{
	\sum_{i = 1}^J \int_{\mathbb{T}^d}
	\left(
	\frac{1}{2}|\nabla \phi_i(x) + p|^2 
	-
	\sum_{j = 1}^J r_{ij}(x)
	\right) \,
	\dd \mu_i(x)
	\\+
	\sum_{i, j = 1}^J
	\int_{\mathbb{T}^d} \pi_{ij}(x) \sqrt{\overline{\mu}_i(x) \overline{\mu}_j(x)}
	e^{\psi_j(x) + \psi_i(x)} \cosh{(\phi(x,j) - \phi(x,i))}
	\,
	\dd x
	\bigg\},
	\end{multline*}
	where $\pi_{ij}(x) = r_{ij}(x) e^{-2\psi_i(x)}$, the infimum is taken over vectors of functions $\phi_i = \phi(\cdot,i) \in C^2(\mathbb{T}^d)$, and $\dd \mu_i(x) = \overline{\mu}_i(x) \dd x$.
	\medskip
	
	Let~$\mu\in\mathbf{P}$. We show that~$K_p(\mu) = K_{-p}(\mu)$, which implies~$\mathcal{H}(p)=\mathcal{H}(-p)$. The sum in which the~$\cosh(\cdot)$ terms appear is symmetric in the sense that
	\begin{equation*}
	C(\phi):=
	\sum_{i, j = 1}^J \int_{\mathbb{T}^d} \pi_{ij}(x)
	\sqrt{\overline{\mu}_i(x) \overline{\mu}_j(x)}
	e^{\psi_j(x) + \psi_i(x)} \cosh{(\phi(x,j) - \phi(x,i))}
	\,
	\dd x
	\end{equation*}
	satisfies $C(\phi) = C(-\phi)$. The bijective transformation $\phi \to (-\phi)$ leaves the infimum in~$K_p(\mu)$ invariant, and hence symmetry of~$C(\phi)$ implies the claimed symmetry~$K_p(\mu) = K_{-p}(\mu)$.
\end{proof}
With a similar analysis, we can study the behaviour of molecular motors under external forces. Let $(X^n, I^n)$ be the stochastic process of Theorem~\ref{thm:results:LDP_cont_MM_Limit_I} in dimension~$d = 1$ with drift~$b^i(y)=F - \psi'(y,i)$, where~$F$ is a constant (modeling an external force) and~$\psi \in C^\infty(\mathbb{T})$ is a smooth periodic potential. The process~$(X^n, I^n)$ is~$\mathbb{T} \times \{1,\dots,J\}$-valued and satisfies
$$
\dd X^n_t = (F - \psi'(n X^n_t, I^n_t))\, \dd t + \frac{1}{\sqrt{n}} \,\dd B_t,
$$
where $I^n_t$ a jump process on $\{1,\dots,J\}$ with jump rates~$n r_{ij}\left(nx\right)$.
Under detailed balance, one can show with arguments similar as above that the Hamiltonian for this process is symmetric around $(-F)$.
Since~$\mathcal{H}(0) = 0$ and~$\mathcal{H}(p)$ is strictly convex, this means that the model predicts a positive force-velocity feedback under detailed balance:~$F > 0$ implies~$\partial_p\mathcal{H}(0) > 0$, and~$F < 0$ implies~$\partial_p\mathcal{H}(0) < 0$. 
\begin{theorem}[Separation of time scales implies a symmetric Hamiltonian]
\label{thm:results:symmetry_limit_II}
	Let the stochastic process~$(X^n_t, I^n_t)$ of Definition~\ref{def:intro:cont_MM}, with~$b^i=-\nabla\psi^i$, satisfy the assumptions of Theorem \ref{thm:results:LDP_cont_MM_Limit_II}. Suppose in addition that the rates $r_{ij}(\cdot)$ are constant on $\mathbb{T}^d$. Then $\overline{\mathcal{H}}(p) = \overline{\mathcal{H}}(-p)$, where $\overline{\mathcal{H}}(p)$ is the Hamiltonian in Theorem \ref{thm:results:LDP_cont_MM_Limit_II}.
\end{theorem}
Since the derivation of the required formula for~$\overline{\mathcal{H}}(p)$ is similar to the derivation of~$\mathcal{H}(p)$, we omit the details and only give a sketch of the argument here.
\begin{proof}[Sketch of proof of Theorem~\ref{thm:results:symmetry_limit_II}]
	The principal eigenvalue $\overline{\mathcal{H}}(p)$ is given by
	$$
	\overline{\mathcal{H}}(p)
	=
	\sup_{\mu \in \mathbf{P}}
	\left[
	K_p(\mu) - \mathcal{R}(\mu)
	\right],\quad
	K_p(\mu) = \inf_{\varphi\in C^\infty(\mathbb{T}^d)} \frac{1}{2} \int_{\mathbb{T}^d} | \nabla \varphi + p|^2 \, \dd \mu,
	$$
	with~$\mathbf{P}$ and~$\mathcal{R}$ specified below.
	The bijective transformation $\varphi \to (-\varphi)$ leaves the infimum in~$K_p(\mu)$ invariant, and therefore we have~$K_p(\mu) = K_{-p}(\mu)$ for all~$\mu\in\mathbf{P}$. This implies~$\overline{\mathcal{H}}(p) = \overline{\mathcal{H}}(-p)$.
	\medskip
	
	In the formula for~$\overline{\mathcal{H}}(p)$, the set of probability measures~$\mathbf{P} \subset \mathcal{P}(\mathbb{T}^d)$ is
	\[
	\mathbf{P}
	=
	\left\{ \mu \in \mathcal{P}(\mathbb{T}^d) \, : \, \mu \ll \dd x
	\, \text{ and }\dd \mu = \overline{\mu} \dd x \,
	\text{ with } \nabla 
	\left(
	\log \overline{\mu} 
	\right) 
	\in L^2_{\mu}(\mathbb{T}^d) \right\}.
	\]
	The map~$\mathcal{R}$ is the relative Fisher information; with the stationary measure~$\nu$ of the jump process on~$\{1,\dots,J\}$ with rates~$r_{ij}$,
	\begin{equation*}
	\mathcal{R}(\mu) = \frac{1}{8}\int_{\mathbb{T}^d}\left|\nabla\log\left(\frac{\overline{\mu}}{e^{-2\overline{\psi}}}\right)\right|^2\,\dd\mu,\quad \overline{\psi}(x)=\sum_i \nu_i\,\psi_i(x).
	\end{equation*}
\end{proof}
\section{Proof of large-deviation principle for switching Markov processes}
\label{sec:proof-of-LDP}
The main point of this section is to prove Theorem~\ref{thm:results:LDP_switching_MP}, the large-deviation principle for switching Markov processes in a periodic setting. The proof is based on a connection between large deviations and Hamilton-Jacobi equations that we first make explicit in Section~\ref{sec:strategy-LDP-switching-MP} by adapting Theorems of~\cite{FengKurtz2006} to our setting.
\subsection{Strategy of proof}
\label{sec:strategy-LDP-switching-MP}
\emph{Viscosity solutions and comparison principle.} We adapt~\cite[Definitions~6.1 and~7.1]{FengKurtz2006} to the compact setting. For a Banach space~$B$, we identify operators with graphs~$H\subseteq B\times B$, with domain~$\mathcal{D}(H):=\{f\,:\exists\,(f,g)\in H\}$ and range~$\mathcal{R}(H):=\{g\,:\,\exists (f,g)\in H\}$, and refer to them as \emph{multivalued operators}. For the following definition,~$E$ and~$E'$ are compact Polish spaces,~$B(E\times E')$ is the set of measurable and bounded functions on~$E\times E'$, equipped with the uniform norm, and~$M(E\times E')$ is the set of measurable functions.
\begin{definition}[Viscosity solutions]
	Let $H \subseteq C(E) \times M(E\times E^\prime)$ be a multivalued operator with domain~$\mathcal{D}(H) \subseteq C(E)$. Let~$h\in C(E)$ and~$\tau>0$.
	\begin{itemize}
		\item[i)] A function~$u_1:E\to\mathbb{R}$ is a viscosity subsolution of~$(1 - \tau H) u = h$ if it is bounded and upper semicontinuous, and if for all~$(f,g)\in H$ there exists a point~$(x,z^\prime)\in E\times E^\prime$ such that
		\[
		(u_1-f)(x)=\sup(u_1-f)\quad\text{and}\quad u_1(x)-\tau g(x,z^\prime)-h(x)\leq 0.
		\]
		\item[ii)] A function~$u_2:E\to\mathbb{R}$ is a viscosity supersolution of $(1 - \tau H) u = h$ if it is bounded and lower semicontinuous, and if for all $(f,g)\in H$ there exists a point $(x,z^\prime)\in E\times E^\prime$ such that
		\[
		(f-u_2)(x)=\sup(f-u_2)\quad\text{and}\quad u_2(x)-\tau g(x, z^\prime)-h(x)\geq 0.
		\]
		\item[iii)] A function~$u_1:E\to\mathbb{R}$ is a strong viscosity subsolution of $(1 - \tau H) u = h$ if it is bounded and upper semicontinuous, and if for all $(f,g) \in H$ and $x\in E$, whenever
		\[
		(u_1-f)(x)=\sup(u_1-f),
		\]
		then there exists a $z^\prime\in E^\prime$ such that
		\[
		u_1(x)-\tau g(x,z^\prime)-h(x)\leq 0.
		\]
		Similarly for strong viscosity supersolutions.
	\end{itemize}
	A function $u\in C(E)$ is called a viscosity solution of $(1 - \tau H) u = h$ if it is both a viscosity sub- and supersolution.\qed
	\label{def:appendix:viscosity_solutions_multivalued_op}
\end{definition}
Let us briefly highlight the adaptations we made with respect to~\cite{FengKurtz2006}. First, formulating viscosity solutions via sequences as in~\cite[Definition~7.1]{FengKurtz2006} is only required when working with non-compact spaces, while in the context of this paper we only work in compact spaces. Second, the product space~$E\times E'$ in this paper corresponds to the set~$E'$ in~\cite{FengKurtz2006}.
\begin{definition}[Comparison Principle]
	The \emph{comparison principle} holds for viscosity sub- and supersolutions of~$(1-\tau H)u=h$ if for any viscosity subsolution~$u_1$ and viscosity supersolution~$u_2$, we have~$u_1\leq u_2$ on~$E$. \qed
	\label{def:appendix:CP_single_valued_operator}
\end{definition}
If the comparison principle holds, then viscosity solutions are unique, since two viscosity solutions~$u,v$ satisfy~$u\leq v$ and~$v\leq u$.
\smallskip

\noindent
\emph{A general large-deviation theorem.} Just as in Theorem~\ref{thm:results:LDP_switching_MP}, we work with compact Polish spaces~$E_n$,~$E$ and~$E^\prime$ that are related via continuous embeddings~$\eta_n$ and~$\eta_n^\prime$ by
$$
\begin{tikzcd}[row sep = tiny]
&	E \times E^\prime \arrow[dd, "\mathrm{proj}_1"]	\\
E_n \arrow[ur, "(\eta_n {,} \eta_n^\prime)"] \arrow[dr, "\eta_n"']	&	\\
&	E	
\end{tikzcd}
$$
such that for any $x \in E$, there exist $x_n \in E_n$ such that $\eta_n(x_n) \to x$ as $n \to \infty$.
The following Theorem is an adaptation of~\cite[Theorem 7.18]{FengKurtz2006} to our setting. This adaptation is obtained by collecting in one place assumptions that are mentioned in several places in~\cite{FengKurtz2006}, and specializing them to the compact setting.
\begin{theorem}
	Let $L_n$ be the generator of an $E_n$-valued process $Y^n$, and let $H_n$ be the nonlinear generators defined by $H_n f = \frac{1}{n} e^{-nf} L_n e^{nf}$. Let the compact Polish spaces $E_n, E$ and $E'$ be related as in the above diagram. In addition, suppose:
	\begin{enumerate}[(i)]
		\item(Condition 7.9 of \cite{FengKurtz2006} on the state spaces)
		\label{item:condition79:FengKurtz}
		There exists an index set $Q$ and approximating state spaces $A_n^q \subseteq E_n$, $q \in Q$, such that the following holds:
		\begin{enumerate}[(a)]
			\item For $q_1, q_2 \in Q$, there exists $q_3 \in Q$ such that $A_n^{q_1} \cup A_n^{q_2} \subseteq A_n^{q_3}$.
			\item For each $x \in E$, there exists $q \in Q$ and $y_n \in A^q_n$ such that $\eta_n(y_n) \to x$ as $n \to \infty$.
			\item For each $q \in Q$, there exist compact sets $K_1^q \subseteq E$ and $K_2^q \subseteq E \times E^\prime$ such that
			\[
			\sup_{y \in A^q_n} \inf_{x \in K_1^q} d_E(\eta_n(y), x) \xrightarrow{n\to\infty} 0,
			\]
			and
			\[
			\sup_{y \in A^q_n} \inf_{(x,z) \in K_2^q} 
			\left[ d_E(\eta_n(y),x)) + d_{E'}(\eta_n'(y),z) \right] \xrightarrow{n\to\infty} 0.
			\]
			\item For each compact $K \subseteq E$, there exists $q \in Q$ such that
			$
			K \subseteq \liminf \eta_n(A_n^q).
			$
		\end{enumerate}
		\item 
		(Convergence Condition 7.11 of \cite{FengKurtz2006}) 
		\label{item:condition711:FengKurtz}
		There exist multivalued operators $H_\dagger, H_{\ddagger} \subseteq C(E) \times C(E \times E^\prime)$ which are the limit of the $H_n$'s in the following sense:
		\begin{enumerate}[(a)]
			\item
			For each $(f,g) \in H_\dagger$, there exist $f_n \in \mathcal{D}(H_n)$ such that
			\[
			\sup_n \left(
			\sup_{x \in E_n} |f_n (x)| + \sup_{x \in E_n} |H_n f_n (x)|
			\right)
			< \infty,
			\]
			and for each $q\in Q$,~$\lim_{n \to \infty} \sup_{y \in A_n^q}|f_n(y) - f(\eta_n(y))| = 0.$
			Furthermore, for each $q \in Q$ and every sequence $y_n \in A_n^q$ such that $\eta_n (y_n) \to x \in E$ and $\eta'_n(y_n) \to z^\prime \in E^\prime$, we have~$\limsup_{n \to \infty} H_n f_n (y_n)
			\leq
			g(x,z^\prime)$.
			\item 
			For each $(f,g) \in H_\ddagger$, there exist $f_n \in \mathcal{D}(H_n)$ (not necessarily the same as above in (a)) such that
			\[
			\sup_n \left(
			\sup_{x \in E_n} |f_n (x)| + \sup_{x \in E_n} |H_n f_n (x)|
			\right)
			< \infty,
			\]
			and for each $q\in Q$,~$\lim_{n \to \infty} \sup_{y \in A_n^q}|f_n(y) - f(\eta_n(y))| = 0$.
			Furthermore, for each $q \in Q$ and every sequence $y_n \in E_n$ such that $\eta_n (y_n) \to x \in E$ and $\eta'_n(y_n) \to z^\prime \in E^\prime$, we have~$\liminf_{n \to \infty} H_n f_n (y_n)
			\geq
			g(x,z^\prime)$.
		\end{enumerate}
		\item (Comparison principle) 
		\label{item:CP:FengKurtz}
		For each $h \in C(E)$ and $\tau > 0$, the comparison principle holds for viscosity subsolutions of $(1 - \tau H_\dagger) u = h$ and viscosity supersolutions of $(1 - \tau H_\ddagger) u = h$.
	\end{enumerate}
	Let $X^n_t := \eta_n (Y^n_t)$ be the corresponding $E$-valued process. Suppose that $\{X^n(0)\}_{n \in \mathbb{N}}$ satisfies a large-deviation principle in $E$ with rate function $\mathcal{I}_0 : E \to [0,\infty]$.
	\medskip
	
	Then
	$\{X^n\}_{n \in \mathbb{N}}$
	satisfies the large-deviation principle with a rate function $\mathcal{I} : C_E[0,\infty) \to [0,\infty]$. Furthermore, there exists a semigroup~$V(t):C(E)\to C(E)$ with which the rate function is given by 
	\begin{equation}\label{eq:rate-function:FengKurtz}
	\mathcal{I}(x) = \mathcal{I}_0(x(0)) + \sup_{k\in\mathbb{N}}\sup_{(t_1,\dots,t_k)} \sum_{i=1}^k \mathcal{I}_{t_i-t_{i-1}}(x(t_i)|x(t_{i-1})),
	\end{equation}
	where for~$z,y\in E$,
	\begin{equation}\label{eq:rate-function-2:FengKurtz}
	\mathcal{I}_t(z|y) = \sup_{f\in C(E)}\left[f(z)-V(t)f(y)\right].
	\end{equation}
	\label{thm:appendix:LDP_via_CP:Jin_LDP_thm}
\end{theorem}
The semigroup~$V(t)$ is defined via the Crandall-Liggett Theorem---for details we refer to~\cite[Chapter~5]{FengKurtz2006}. 
\subsection{Proof of Theorem~\ref{thm:results:LDP_switching_MP}}
We prove Theorem~\ref{thm:results:LDP_switching_MP} by verifying the conditions of Theorem~\ref{thm:appendix:LDP_via_CP:Jin_LDP_thm}, which are convergence of nonlinear generators (Proposition~\ref{prop:LDP_switching_MP:convergence_condition_sufficient}) and the comparison principle (Proposition~\ref{prop:LDP_switching_MP:comparison_principle}). The rest of this section below the proof of Theorem~\ref{thm:results:LDP_switching_MP} is devoted to proving the propositions. We point out that the main challenge is to prove the comparison principle using only~\ref{MM:item:T1} and~\ref{MM:item:T2} of Theorem~\ref{thm:results:LDP_switching_MP}.
\begin{proposition}
	In the setting of Theorem~\ref{thm:results:LDP_switching_MP}, condition~(i) of Theorem~\ref{thm:appendix:LDP_via_CP:Jin_LDP_thm} is satisfied. Let $H \subseteq C^1(\mathbb{T}^d) \times C(\mathbb{T}^d \times E^\prime)$ be a multivalued operator satisfying~\ref{MM:item:T1}. Then~$H$ satisfies the convergence condition~\ref{item:condition711:FengKurtz} of Theorem~\ref{thm:appendix:LDP_via_CP:Jin_LDP_thm}.
	\label{prop:LDP_switching_MP:convergence_condition_sufficient}
\end{proposition}	
\begin{proposition}
	In the setting of~Theorem \ref{thm:results:LDP_switching_MP}, let $H \subseteq C^1(\mathbb{T}^d) \times C(\mathbb{T}^d \times E^\prime)$ be a multivalued operator satisfying conditions~\ref{MM:item:T1} and~\ref{MM:item:T2}. Then for $\tau > 0$ and $h \in C(\mathbb{T}^d)$, the comparison principle is satisfied for viscosity sub- and supersolutions of~$(1 - \tau H) u =h$.
	\label{prop:LDP_switching_MP:comparison_principle}
\end{proposition}
\begin{proof}[Proof of Theorem \ref{thm:results:LDP_switching_MP}]
	By Proposition~\ref{prop:LDP_switching_MP:convergence_condition_sufficient}, conditions (i) and (ii) of Theorem~\ref{thm:appendix:LDP_via_CP:Jin_LDP_thm} hold with the single operator $H = H_\dagger = H_\ddagger$. By Proposition~\ref{prop:LDP_switching_MP:comparison_principle}, the comparison principle is satisfied for $(1 - \tau H)u = h$, and hence condition (iii) of Theorem~\ref{thm:appendix:LDP_via_CP:Jin_LDP_thm} holds with a single operator $H = H_\dagger = H_\ddagger$. Therefore the large-deviation principle follows by Theorem~\ref{thm:appendix:LDP_via_CP:Jin_LDP_thm}.
\end{proof}
\begin{proof}[Proof of Proposition \ref{prop:LDP_switching_MP:convergence_condition_sufficient}]
	We recall that with~$E_n = E_n^X \times \{1,\dots,J\}$ and~$\iota_n : E^X_n \to \mathbb{T}^d$ of Condition~\ref{condition:results:general_setting}, the state spaces are related as in the following diagram,
	$$
	\begin{tikzcd}[row sep = tiny]
	&	\mathbb{T}^d \times E^\prime \arrow[dd, "\mathrm{proj}_1"]	\\
	E_n \arrow[ur, "(\eta_n {,} \eta_n^\prime)"] \arrow[dr, "\eta_n"']	&	\\
	&	\mathbb{T}^d	
	\end{tikzcd}
	$$
	where~$\eta_n : E_n \to \mathbb{T}^d$ is defined by~$\eta_n(x,i) = \iota_n(x)$ and~$\eta_n' : E_n \to E'$ is a continuous map.
	In the notation of Theorem \ref{thm:appendix:LDP_via_CP:Jin_LDP_thm}, we have $E = \mathbb{T}^d$. For verifying the general condition (i) of Theorem \ref{thm:appendix:LDP_via_CP:Jin_LDP_thm} on the approximating state spaces $A_n^q$, we take the singleton $Q = \{q\}$ and set $A_n^q := E_n$. Then part (a) holds, and parts (b) and~(d) are a consequence of Condition~\ref{condition:results:general_setting} on $E_n$, which says that for any $x \in \mathbb{T}^d$, there exist $ x_n \in E_n^X$ such that $\iota_n(x_n) \to x$. Part (c) follows by taking the compact sets $K_1^q := \mathbb{T}^d$ and $K_2^q := \mathbb{T}^d \times E^\prime$.
	\medskip 
	
	We verify the convergence Condition (ii) of Theorem \ref{thm:appendix:LDP_via_CP:Jin_LDP_thm}. By~\ref{MM:item:T1}, part~\ref{MM:item:C2}, there exist $f_n \in \mathcal{D}(H_n)$ such that
	$$
	\| f\circ\eta_n - f_n \|_{L^\infty(E_n)} \xrightarrow{n\to \infty} 0
	\quad\text{and}\quad
	\| H_{f,\varphi}\circ (\eta_n, \eta_n') - H_n f_n \|_{L^\infty(E_n)} \xrightarrow{n\to\infty} 0.
	$$
	With these $f_n$, both conditions (a) and (b) are simultaneously satisfied for the operator $H = H_\dagger = H_\ddagger$, where condition~\ref{MM:item:C1} guarantees that for any point $(x,z') \in \mathbb{T}^d \times E'$ there exist $y_n \in E_n$ such that both $\eta_n(y_n) \to x$ and $\eta_n'(y_n) \to z'$. The required boundedness,
	$$
	\sup_{n \in\mathbb{N}} \left( \sup_{y \in E_n}|f_n(y)| + \sup_{y \in E_n}|H_n f_n(y)|\right) < \infty,
	$$
	follows from the uniform-convergence condition~\ref{MM:item:C2}.
\end{proof}
For proving Proposition \ref{prop:LDP_switching_MP:comparison_principle}, we use two operators~$H_1,H_2$ that are derived from a multivalued limit $H$. Define~$H_1,H_2:C(E)\to M(E)$ by
\begin{equation*}
H_1f(x)
:=
\inf_{\varphi} \sup_{z^\prime \in E^\prime}
H_{f, \varphi}(x,z^\prime)
\quad\text{and}\quad
H_2f(x)
:=
\sup_{\varphi} \inf_{z^\prime \in E^\prime}
H_{f, \varphi}(x,z^\prime),
\end{equation*}
with equal domains~$\mathcal{D}(H_1) = \mathcal{D}(H_2) := \mathcal{D}(H)$. 
Since the images of~$H$ are of the form $H_{f,\varphi}(x,z^\prime) = H_{\varphi}(\nabla f(x),z^\prime)$, the operators $H_1$ and $H_2$ are as well of the form $H_1 f(x) = \mathcal{H}_1(\nabla f(x))$ and $H_2 f(x) = \mathcal{H}_2(\nabla f(x))$, with two maps $\mathcal{H}_1,\mathcal{H}_2 : \mathbb{R}^d \rightarrow \mathbb{R}$.
We prove Proposition \ref{prop:LDP_switching_MP:comparison_principle} with the following Lemmas.
\begin{lemma}[Local operators admit strong solutions]
	Let 
	$
	H \subseteq C^1(\mathbb{T}^d) \times C(\mathbb{T}^d\times E^\prime) 
	$
	be a multivalued limit operator
	satisfying (T1) of Theorem \ref{thm:results:LDP_switching_MP}. Then for any $\tau > 0$ and $h \in C(\mathbb{T}^d)$, viscosity solutions of
	$
	(1 - \tau H) u =h
	$
	coincide with strong viscosity solutions in the sense of Definition \ref{def:appendix:viscosity_solutions_multivalued_op}.
	\label{lemma:LDP_switching_MP:local_op_strong_sol}
\end{lemma}
\begin{lemma}[$H_1$ and $H_2$ are viscosity extensions]
	Let $H$ be a multivalued operator satisfying~\ref{MM:item:T1} and~\ref{MM:item:T2} of Theorem~\ref{thm:results:LDP_switching_MP}. For all $h \in C(\mathbb{T}^d)$ and $\tau > 0$, strong viscosity subsolutions $u_1$ of
	$
	(1 - \tau H) u =h
	$
	are strong viscosity subsolutions of
	$
	(1 - \tau H_1) u =h,
	$
	and strong viscosity supersolutions $u_2$ of
	$
	(1 - \tau H) u = h
	$
	are strong viscosity supersolutions of
	$
	(1 - \tau H_2) u =h.
	$
	\label{lemma:LDP_switching_MP:H_to_H1-H2}
\end{lemma}
\begin{lemma}[$H_1$ and $H_2$ are ordered]
	Let $H$ be a multivalued operator satisfying~\ref{MM:item:T1} and~\ref{MM:item:T2} of Theorem~\ref{thm:results:LDP_switching_MP}. Then~$\mathcal{H}_1(p)\leq \mathcal{H}_2(p)$ for all~$p\in\mathbb{R}^d$.
	\label{lemma:LDP_switching_MP:H1_leq_H2}
\end{lemma}
\begin{proof}[Proof of Proposition \ref{prop:LDP_switching_MP:comparison_principle}]
	Let $u_1$ be a subsolution and $u_2$ be a supersolution of the equation $(1 - \tau H)u = h$. By Lemma~\ref{lemma:LDP_switching_MP:local_op_strong_sol}, $u_1$ is a strong subsolution and $u_2$ a strong supersolution of $(1 - \tau H) u = h$, respectively. By Lemma~\ref{lemma:LDP_switching_MP:H_to_H1-H2}, $u_1$ is a strong subsolution of $(1 - \tau H_1) u = h$, and $u_2$ is a strong supersolution of $H_2$.
	\medskip 
	
	With that, we establish below the inequality
	\begin{equation}
	\max_{\mathbb{T}^d} (u_1 - u_2)
	\leq
	\tau \left[ \mathcal{H}_1(p_\delta) - \mathcal{H}_2(p_\delta) \right]
	+
	h(x_\delta) - h(x_\delta^\prime),
	\label{eq:LDP_switching_MP:CP_visc_ineq}
	\end{equation}
	with some $x_\delta, x_\delta^\prime \in \mathbb{T}^d$ such that $\text{dist}(x_\delta,x_\delta^\prime) \rightarrow 0$ as $\delta \rightarrow 0$, and certain $p_\delta \in \mathbb{R}^d$. Then using that $h\in C(\mathbb{T}^d)$ is uniformly continuous since~$\mathbb{T}^d$ is compact, and that $\mathcal{H}_1(p_\delta) \leq \mathcal{H}_2(p_\delta)$ by Lemma~\ref{lemma:LDP_switching_MP:H1_leq_H2}, we can further estimate as
	$$
	\max_{\mathbb{T}^d} (u_1 - u_2)
	\leq
	h(x_\delta) - h(x_\delta^\prime)
	\leq
	\omega_h(\text{dist}(x_\delta,x_\delta^\prime)),
	$$
	where $\omega_h : [0,\infty) \rightarrow [0,\infty)$ is a modulus of continuity satisfying $\omega_h(r_\delta) \rightarrow 0$ for $r_\delta \rightarrow 0$. Then $(u_1 - u_2) \leq 0$ follows by taking the limit $\delta \rightarrow 0$.
	\medskip
	
	We are left with proving~\eqref{eq:LDP_switching_MP:CP_visc_ineq}. Define $\Phi_\delta : \mathbb{T}^d \times \mathbb{T}^d \rightarrow \mathbb{R}$ by
	$$
	\Phi_\delta(x,x^\prime)
	:=
	u_1(x) - u_2(x^\prime) - \frac{\Psi(x,x^\prime)}{2 \delta},
	$$
	where
	\begin{equation}
	\Psi(x,x^\prime)
	:=
	\sum_{j=1}^d \sin^2\left( \pi(x_j - x_j^\prime) \right),
	\qquad
	\text{ for all } x,x^\prime \in \mathbb{T}^d.
	\label{eq:LDP_switching_MP:dist_function_E}
	\end{equation}
	Then $\Psi \geq 0$, and $\Psi(x,x^\prime) = 0$ holds if and only if $x = x^\prime$, and
	\begin{equation}
	\nabla_1 \left[ \Psi(\cdot,x^\prime)\right](x)
	=
	- \nabla_2 \left[ \Psi(x,\cdot) \right](x^\prime)
	\qquad
	\text{ for all } x,x^\prime \in \mathbb{T}^d.
	\label{eq:LDP_switching_MP:dist_function_asymm_deriv}
	\end{equation}
	By boundedness and upper semicontinuity of $u_1$ and $(-u_2)$, and compactness of $\mathbb{T}^d \times \mathbb{T}^d$, for each $\delta > 0$ there exists a pair $(x_\delta,x_\delta^\prime) \in \mathbb{T}^d \times \mathbb{T}^d$ such that
	$$
	\Phi_\delta(x_\delta,x_\delta^\prime)
	=
	\max_{x,x^\prime} \Phi_\delta(x,x^\prime).
	$$
	Since~$\Phi_\delta(x_\delta, x_\delta) \leq \Phi(x_\delta,x_\delta^\prime)$ and~$u_2$ is bounded, we obtain
	$$
	\Psi(x_\delta,x_\delta^\prime)
	\leq
	2\delta \left( u_2(x_\delta) - u_2(x_\delta^\prime) \right)
	\leq	
	4 \delta \|u_2\|_{L^\infty(\mathbb{T}^d)} = \mathcal{O}(\delta).
	$$
	Hence~$\Psi(x_\delta,x_\delta^\prime) \rightarrow 0$ as~$\delta \rightarrow 0$.
	\medskip
	
	In order to use the sub- and supersolution properties of $u_1$ and $u_2$, introduce the smooth test functions $f^\delta_1$ and $f^\delta_2$ as
	$$
	f_1^\delta(x)
	:=
	u_2(x_\delta^\prime) + \frac{\Psi(x,x_\delta^\prime)}{2\delta}
	\quad
	\text{ and }
	\quad
	f_2^\delta(x^\prime)
	:=
	u_1(x_\delta) - \frac{\Psi(x_\delta,x^\prime)}{2\delta},
	$$
	Then $f_1^\delta, f_2^\delta \in C^\infty(\mathbb{T}^d) \subseteq \mathcal{D}(H)$ are both in the domain of $H$, and hence in the domain of $H_1$ and $H_2$, respectively. Furthermore, $(u_1 - f_1)$ has a maximum at~$x = x_\delta$, and $(f_2 - u_2)$ has a maximum at~$x^\prime = x_\delta^\prime$, by definition of~$(x_\delta,x_\delta^\prime)$ and~$\Phi_\delta$. Since~$u_1$ is a strong subsolution of~$(1 - \tau H_1) u = h$, 
	$$
	u_1(x_\delta) - \tau H_1 f_1^\delta(x_\delta) - h(x_\delta)
	\leq
	0,
	$$
	and since $u_2$ is a strong supersolution of $(1 - \tau H_2) u = h$, 
	$$
	u_2(x_\delta^\prime) - \tau H_2 f_2^\delta(x_\delta^\prime) - h(x_\delta^\prime)
	\geq
	0.	
	$$
	Thereby, we can estimate $\max (u_1 - u_2)$ as
	\begin{align*}
	\max_{\mathbb{T}^d}(u_1 - u_2)
	&\leq
	\Phi_\delta(x_\delta,x_\delta^\prime)	\\
	&\leq
	u_1(x_\delta) - u_2(x_\delta^\prime)	\\
	&\leq
	\tau \left[ H_1 f_1^\delta(x_\delta)) - H_2 f_2^\delta(x_\delta^\prime)\right]
	+
	h(x_\delta) - h(x_\delta^\prime)	\\
	&=
	\tau \left[ \mathcal{H}_1 (\nabla f_1^\delta(x_\delta)) - \mathcal{H}_2 (\nabla f_2^\delta(x_\delta^\prime))\right]
	+
	h(x_\delta) - h(x_\delta^\prime).
	\end{align*}
	By \eqref{eq:LDP_switching_MP:dist_function_asymm_deriv}, 
	$
	\nabla f_1^\delta(x_\delta)
	=
	\nabla f_2^\delta(x_\delta^\prime)
	=:
	p_\delta \in \mathbb{R}^d,
	$
	which establishes \eqref{eq:LDP_switching_MP:CP_visc_ineq}, and thereby finishes the proof.
\end{proof}
The rest of the section, we prove Lemmas \ref{lemma:LDP_switching_MP:local_op_strong_sol}, \ref{lemma:LDP_switching_MP:H_to_H1-H2} and \ref{lemma:LDP_switching_MP:H1_leq_H2}. Regarding Lemma \ref{lemma:LDP_switching_MP:local_op_strong_sol}, a proof for single valued operators is given in~\cite[Lemma 9.9]{FengKurtz2006}.
\begin{proof}[Proof of Lemma \ref{lemma:LDP_switching_MP:local_op_strong_sol}]
	Let $\tau>0$, $h\in C(\mathbb{T}^d)$. We verify that subsolutions are strong subsolutions. Let~$u_1$ be a subsolution of~$(1-\tau H)u=h$ and~$(f,H_{f,\varphi}) \in H$, and let~$x\in \mathbb{T}^d$ be such that~$(u_1-f)(x)=\sup(u_1-f)$. 
	\medskip
	
	The function~$\tilde{f}$ defined by~$\tilde{f}(x^\prime):=\Psi(x^\prime,x)$, with $\Psi(x^\prime,x)$ defined by~\eqref{eq:LDP_switching_MP:dist_function_E}, is smooth and therefore~$\tilde{f}$ is in the domain $\mathcal{D}(H)$. Then $x$ is the unique maximal point of $(u_1-(f+\tilde{f}))$,
	\[
	(u_1-(f+\tilde{f}))(x)=\sup_{\mathbb{T}^d}(u_1-(f+\tilde{f})).
	\]
	Since~$u_1$ is a subsolution, there exists a point~$z'\in E'$ such that
	\[
	u_1(x)-\tau H_{f+\tilde{f},\varphi}(x,z^\prime)-h(x)\leq 0.
	\]
	Using~$\nabla\tilde{f}(x)=0$ and that~$H$ depends only on gradients by~\ref{MM:item:T1}, we obtain
	\[
	H_{f+\tilde{f},\varphi}(x,z^\prime)
	=
	H_{\varphi}\left((\nabla f+\nabla \tilde{f})(x),z^\prime\right)
	=
	H_{\varphi}(\nabla f(x),z^\prime)
	=
	H_{f,\varphi}(x,z^\prime).
	\]
	Hence
	\[
	u_1(x)-\tau H_{f,\varphi}(x,z^\prime)-h(x)\leq 0.
	\]
	Thus~$u_1$ is a strong subsolution. The argument is similar for the supersolution case, where one can use $(-\tilde{f})$.
	\medskip
	
	Vice versa, when given a strong sub- or supersolution $u_1$ or $u_2$, for every $f\in\mathcal{D}(H)$, $(u_1-f)$ and $(f-u_2)$ attain their suprema at some $x_1,x_2\in \mathbb{T}^d$ due to the continuity assumptions on the domain of $H$, the semi-continuity properties of $u_1$ and $u_2$,  and compactness of $\mathbb{T}^d$. By the strong solution properties, the sub- and supersolution inequalities follow.
\end{proof}
\begin{proof}[Proof of Lemma \ref{lemma:LDP_switching_MP:H_to_H1-H2}]
	Let $u_1$ be a strong subsolution of~$(1-\tau H)u=h$, that is for any $(f,H_{f, \varphi})$, if~$(u_1-f)(x)=\sup(u_1-f)$ for a point~$x\in \mathbb{T}^d$, then there exists a point~$z^\prime\in E^\prime$ such that
	\begin{equation}\label{eq:proof:lemma:sub-H-to-sub-H1}
	u_1(x)-\tau H_{f, \varphi}(x,z^\prime)-h(x)\leq 0.
	\end{equation}
	Let $f\in\mathcal{D}(H_1)=\mathcal{D}(H)$ and $x\in \mathbb{T}^d$ be such that $(u_1-f)(x)=\sup(u_1-f)$. For any $\varphi$ there exists a point $z^\prime\in E^\prime$ such that the above subsolution inequality~\eqref{eq:proof:lemma:sub-H-to-sub-H1} holds. Therefore for all~$x$,
	$$
	u_1(x)-h(x)
	\leq 
	\tau\sup_{z^\prime\in E^\prime} H_{f, \varphi}(x,z^\prime).
	$$
	Since the point $x\in \mathbb{T}^d$ is independent of $\varphi$, we obtain
	\[
	u_1(x)-\tau H_1f(x)-h(x)\overset{\text{def}}{=}
	u_1(x)-\tau\inf_{\varphi}\sup_{z^\prime \in E^\prime} H_{f, \varphi}(x,z^\prime)-h(x)\leq 0.
	\]
	The argument is similar for supersolutions.
\end{proof}
\begin{proof}[Proof of Lemma \ref{lemma:LDP_switching_MP:H1_leq_H2}]
	By assumption, for every~$p\in\mathbb{R}^d$ there exists a function~$\varphi_p\in C(E^\prime)$ such that for all $z^\prime\in E^\prime$,
	\[
	H_{\varphi_p}(p,z^\prime)=\mathcal{H}(p).
	\]
	Thus
	\[
	\sup_{z^\prime \in E^\prime} H_{\varphi_p}(p,z^\prime)
	=
	\mathcal{H}(p)
	=
	\inf_{z^\prime\in E^\prime}H_{\varphi_p}(p, z^\prime).
	\]
	Taking the infimum and supremum over~$\varphi$, we find
	\begin{align*}
	\mathcal{H}_1(p)
	&=
	\inf_{\varphi}\sup_{z^\prime}H_{\varphi}(p,z^\prime)\\
	&\leq
	\sup_{z^\prime}H_{\varphi_p}(p,z^\prime)
	=
	\mathcal{H}(p)
	=\inf_{z^\prime}H_{\varphi_p}(p,z^\prime)\\
	&\leq
	\sup_{\varphi}\inf_{z^\prime}H_{\varphi}(p,z^\prime)
	=
	\mathcal{H}_2(p),
	\end{align*}
	which finishes the proof.
\end{proof}
\section{Proof of action-integral representation}
\label{section:action_integral}
In this section we prove Theorem~\ref{thm:results:action_integral_representation}, the action-integral representation of the rate function of Theorem~\ref{thm:results:LDP_switching_MP}, by following the strategy outlined in~\cite[Chapter~8]{FengKurtz2006}. We first briefly summarize the strategy in Section~\ref{sec:strategy-action-integral}, specialized to our setting.
\subsection{Strategy of proof}
\label{sec:strategy-action-integral}
Let~$\mathcal{H}=\mathcal{H}(p)$ be the Hamiltonian of Theorem~\ref{thm:results:action_integral_representation} and let~$\mathcal{L}=\mathcal{L}(v)$ be the associated Lagrangian defined by
\begin{equation}\label{eq:proof-action-integral:Lagrangian}
    \mathcal{L}(v) := \sup_{p \in \mathbb{R}^d} \left[p\cdot v - \mathcal{H}(p)\right]
\end{equation}
Define~$V_{\mathrm{NS}}(t) : C(\mathbb{T}^d) \to C(\mathbb{T}^d)$ by
\begin{equation}\label{eq:proof-action-integral:nisio-semigroup}
	V_{\text{NS}}(t) f(x)
	=
	\sup_{ \substack{ \gamma \in \mathrm{AC}_{\mathbb{T}^d}[0,\infty) \\ \gamma(0) = x}}
	\left[
	f(\gamma(t)) - \int_0^t \mathcal{L}(\partial_s\gamma(s)) \, \dd s
	\right],
\end{equation}
where~$\mathrm{AC}_{\mathbb{T}^d}[0,\infty)$ is the set of absolutely continuous paths in the torus. The map~$V_{\mathrm{NS}}(t)$ is the \emph{Nisio semigroup} with cost function~$\mathcal{L}$. In Definition~8.1 and Equation~(8.10) in~\cite{FengKurtz2006}, the Nisio semigroup is defined by means of relaxed controls in order to cover a general class of possible cost functions. Since the Lagrangian~$\mathcal{L}(v)$ is convex, the semigroup~$V_{\mathrm{NS}}(t)$ equals the semigroup given in~(8.10) of~\cite{FengKurtz2006}, which can be seen by using that~$\lambda_s = \delta_{\partial_sx(s)}$ is an admissible control and by applying Jensen's inequality. Such an argument is given for example in Theorem 10.22 in~\cite{FengKurtz2006}.
\smallskip

The rate function~$\mathcal{I}$ of Theorem~\ref{thm:results:LDP_switching_MP} is given in terms of a limiting semigroup~$V(t)$ as shown in equations~\eqref{eq:rate-function:FengKurtz} and~\eqref{eq:rate-function-2:FengKurtz}. The desired action-integral representation follows if the semigroup~$V(t)$ of Theorem~\ref{thm:results:LDP_switching_MP} is equal to the Nisio semigroup~$V_{\text{NS}}(t)$ defined by~\eqref{eq:proof-action-integral:nisio-semigroup}. In~\cite[Chapter~8]{FengKurtz2006}, the equality of semigroups is traced back to conditions on their generators. In our case, the generator of the limiting seimgroup is the limiting multivalued operator~$H$ of Theorem~\ref{thm:results:LDP_switching_MP}, and the generator of the Nisio semigroup is an operator~$\mathbf{H}$ defined by the Hamiltonian~$\mathcal{H}(p)$.
We summarize in Proposition~\ref{prop:action_integral:bold_H} below that the generators satisfy the required conditions of~\cite[Chapter~8]{FengKurtz2006} and show that these conditions suffice to prove the action-integral representation.
\subsection{Proof of Theorem~\ref{thm:results:action_integral_representation}}
In this section, we first prove Theorem~\ref{thm:results:action_integral_representation} by means of Proposition~\ref{prop:action_integral:bold_H} below. The rest of the section is then devoted to proving Proposition~\ref{prop:action_integral:bold_H}.
\begin{proposition}
Under the same assumptions of Theorems~\ref{thm:results:LDP_switching_MP} and~\ref{thm:results:action_integral_representation}, define the operator $\mathbf{H} : \mathcal{D}(\mathbf{H}) \subseteq C^1(\mathbb{T}^d) \to C(\mathbb{T}^d)$ on the domain $\mathcal{D}(\mathbf{H}) = \mathcal{D}(H)$ by setting $\mathbf{H} f(x) := \mathcal{H}(\nabla f(x))$. Let~$\tau>0$ and~$h\in C(\mathbb{T}^d)$. Then:
	\begin{enumerate}[(i)]
		\item
		The Lagrangian~\eqref{eq:proof-action-integral:Lagrangian} and the operator~$\mathbf{H}$ satisfy Conditions~8.9, 8.10 and 8.11 of \cite{FengKurtz2006}, with the set of controls~$U=\mathbb{R}^d$, operator~$Af(x,u)=\nabla f(x)\cdot u$, cost function~$L(x,u)=\mathcal{L}(u)$, and~$\mathbf{H}_\dagger=\mathbf{H}_\ddagger=\mathbf{H}$.
		\item
		The comparison principle (Definition~\ref{def:appendix:CP_single_valued_operator}) holds for viscosity sub- and supersolutions of~$(1 - \tau \mathbf{H}) u = h$.
		\item
		Every viscosity solution~$u$ of~$(1 - \tau H) u = h$ is also a viscosity solution of~$(1 - \tau \mathbf{H}) u = h$.
	\end{enumerate}
	\label{prop:action_integral:bold_H}
\end{proposition}	
\begin{proof}[Proof of Theorem \ref{thm:results:action_integral_representation}]
Let~$V(t)$ be the semigroup obtained in Theorem~\ref{thm:results:LDP_switching_MP} and let~$V_{\text{NS}}(t)$ bet the Nisio semigroup~\eqref{eq:proof-action-integral:nisio-semigroup}. We shall verify that~$V(t)=V_{\text{NS}}(t)$. Then by~\cite[Theorem~8.14]{FengKurtz2006}, the rate function of Theorem~\ref{thm:results:LDP_switching_MP} (given by~\eqref{eq:rate-function:FengKurtz}) satisfies the control representation~(8.18) of \cite{FengKurtz2006}. The action-integral representation follows from this control representation by applying Jensen's inequality.
\smallskip

By~\cite[Theorem~8.27]{FengKurtz2006}, we obtain~$V_{\text{NS}}(t)=\mathbf{V}(t)$, where the semigroup~$\mathbf{V}(t)$ is defined by
\begin{equation}
    \mathbf{V}(t)
	=
	\lim_{m \to \infty}
	\left[
	\left( 1 - \frac{t}{m} \mathbf{H} \right)^{-1}
	\right]^m.
\end{equation}
The conditions of Theorem~8.27 are satisfied since Conditions~8.9, 8.10 and~8.11 of~\cite{FengKurtz2006} are satisfied by Item~(i), and since the comparison principle holds by Item~(ii).
\smallskip

By~\cite[Corollary~8.29]{FengKurtz2006}, we obtain~$V(t)=\mathbf{V}(t)$. The conditions of Corollary~8.29 are satisfied: Item~(iii) above corresponds to Item~a) of Corollary~8.29, the conditions of~\cite[Theorem~6.14]{FengKurtz2006} are satisfied under the assumptions of our Theorem~\ref{thm:results:LDP_switching_MP}, the conditions of~\cite[Theorem~8.27]{FengKurtz2006} are satisfied for the same reasons as mentioned above, and~$D_\alpha=\mathcal{D}(H)$.
\end{proof}
\begin{proof}[Proof of (i) in Proposition \ref{prop:action_integral:bold_H}]
	We first show that the following Items~\ref{item:proof-action-integral:item-a},~\ref{item:proof-action-integral:item-b},~\ref{item:proof-action-integral:item-c} imply Conditions 8.9, 8.10 and 8.11 of \cite{FengKurtz2006}, which are formulated in order to cover a more general and non-compact setting.
	\begin{enumerate}[(a)]
		\item \label{item:proof-action-integral:item-a} The function $\mathcal{L}:\mathbb{R}^d\rightarrow[0,\infty]$ is lower semicontinuous and for every $C \geq 0$, the level set
		$
		\{v\in \mathbb{R}^d\,:\,\mathcal{L}(v)\leq C\}
		$
		is relatively compact in $\mathbb{R}^d$.
		\item \label{item:proof-action-integral:item-b} For all $f\in\mathcal{D}(H)$ there exists a right continuous, nondecreasing function $\psi_f:[0,\infty)\rightarrow[0,\infty)$ such that for all $(x_0,v)\in \mathbb{T}^d \times \mathbb{R}^d$,
		\[
		|\nabla f(x_0)\cdot v|\leq \psi_f(\mathcal{L}(v))\qquad \text{and} \qquad
		\lim_{r\rightarrow\infty}\frac{\psi_f(r)}{r}=0.
		\]
		\item \label{item:proof-action-integral:item-c} For each $x_0\in E$ and every $f\in\mathcal{D}(\mathbf{H})$, there exists an absolutely continuous path $x : [0,\infty) \to \mathbb{T}^d$ such that
		\begin{equation}
		\int_0^t \mathcal{H}(\nabla f (x(s))) \, ds
		=
		\int_0^t \left[
		\nabla f(x(s)) \cdot \dot{x}(s) - \mathcal{L}(\dot{x}(s))
		\right] \, ds.
		\label{eq:action_integral:optimal_path_for_H}
		\end{equation}
	\end{enumerate}
	Regarding Items~(1)-(5) of~\cite[Condition 8.9]{FengKurtz2006}, the operator $A f(x,v) := \nabla f(x) \cdot v$ defined on the domain $\mathcal{D}(A) = \mathcal{D}(H)$ satisfies Item~(1). For Item~(2), we can take $\Gamma = \mathbb{T}^d \times \mathbb{R}^d$, and for $x_0 \in \mathbb{T}^d$, take the pair $(x,\lambda)$ with $x(t) = x_0$ and $\lambda(dv \times dt) = \delta_{0} (dv) \times dt$. Item~(3) is a consequence of the above Item~\ref{item:proof-action-integral:item-a}. Item~(4) holds since~$\mathbb{T}^d$ is compact. Item~(5) is implied by the above Item~\ref{item:proof-action-integral:item-b}. Condition~8.10 is implied by Condition~8.11 and the fact that $\mathbf{H}1 = 0$, see Remark 8.12 (e) in \cite{FengKurtz2006}. Finally, Condition~8.11 is implied by the above Item~\ref{item:proof-action-integral:item-c}, with the control $\lambda(dv \times dt) = \delta_{\partial_t{x}(t)}(dv) \times dt$.
	\medskip
	
	We turn to verifying Items~\ref{item:proof-action-integral:item-a}, \ref{item:proof-action-integral:item-b} and \ref{item:proof-action-integral:item-c}. Since $\mathcal{H}(0) = 0$, we have $\mathcal{L} \geq 0$. The Legendre-transform $\mathcal{L}$ is convex, and lower semicontinuous since the map $\mathcal{H}(p)$ is convex and finite-valued, hence in particular continuous. For $C \geq 0$, we prove that the set $\{v\in\mathbb{R}^d\,:\,\mathcal{L}(v)\leq C\}$ is bounded, and hence is relatively compact. For any $p \in \mathbb{R}^d$ and $v \in \mathbb{R}^d$, we have $p \cdot  v\leq \mathcal{L}(v) + \mathcal{H}(p)$. Thereby, if $\mathcal{L}(v)\leq C$, then
	$
	|v|
	=
	\sup_{|p|=1} p \cdot v 
	\leq
	\sup_{|p|=1} \left[
	\mathcal{L}(v) + \mathcal{H}(p)
	\right]
	\leq
	C + C_1,
	$
	where $C_1$ exists due to continuity of $\mathcal{H}$. Then for $R := C + C_1$,
	$
	\{ v \, : \, \mathcal{L}(v) \leq C \} \subseteq 
	\{ v \, : \, |v| \leq R\},
	$
	thus $\{\mathcal{L}\leq C\}$ is a bounded subset in $\mathbb{R}^d$.
	\medskip
	
	Item~\ref{item:proof-action-integral:item-b} can be proven as in~\cite[Lemma 10.21]{FengKurtz2006}. We give the proof here. Let~$f\in\mathcal{D}(H)$. There exists a constant~$C_f$ such that for all~$(x_0,v)$, we have
	\[
	|\nabla f(x_0)\cdot v| \leq C_f\cdot |v|.
	\]
	For~$s\geq 0$, define the map~$\varphi(s)$ by
	\[
	\varphi(s) := s \inf_{|v|\geq s} \frac{\mathcal{L}(v)}{|v|}.
	\]
	Let~$\psi_f(r):=C_f\cdot \varphi^{-1}(r)$ with~$\varphi^{-1}(r) = \inf\{w\,:\,\varphi(w)\geq r\}$. 
	By monotonicity of~$\varphi$,
	\[
	\varphi(C_f^{-1}|\nabla f(x_0)\cdot v|) \leq \varphi(|v|) \leq \mathcal{L}(v).
	\]
	Hence by monotonicity of~$\psi_f$, we find~$|\nabla f(x_0)\cdot v| \leq \psi_f(\mathcal{L}(v))$. The map~$\mathcal{L}(v)$ is superlinear, because~$\mathcal{H}(p)$ is convex. Therefore~$s^{-1}\varphi(s)\to+\infty$ as~$s\to \infty$, and consequently~$r^{-1}\psi_f(r)\to 0$ as~$r\to \infty$.
	\medskip
	
	We finish the proof by verifying Item~\ref{item:proof-action-integral:item-c}. This is shown in~\cite[Lemma 3.2.3]{Kraaij2016} under the assumption of continuous differentiability of $\mathcal{H}(p)$, by solving a differential equation with a globally bounded vectorfield. Here, we verify Item~\ref{item:proof-action-integral:item-c} under the milder assumption of convexity of~$\mathcal{H}(p)$ by solving a suitable subdifferential equation. For~$p_0 \in \mathbb{R}^d$, define the subdifferential~$\partial \mathcal{H}(p_0)$ at~$p_0$ as the set
	$$
	\partial \mathcal{H}(p_0)
	:=
	\{\xi \in \mathbb{R}^d \; | \; \forall p \in \mathbb{R}^d\;: 
	\mathcal{H}(p)
	\geq
	\mathcal{H}(p_0) + \langle \xi, p - p_0\rangle\}.
	$$
	We shall solve for any $f \in C^1(\mathbb{T}^d)$ the subdifferential equation~$\dot{x} \in \partial \mathcal{H}(\nabla f(x))$. This means we show that for any initial condition~$x_0\in \mathbb{T}^d$, there exists an absolutely continuous path $x:[0,\infty)\rightarrow \mathbb{T}^d$ satisfying both~$x(0)=x_0$ and~$\dot{x}(t) \in \partial \mathcal{H}(\nabla f(x(t)))$ almost everywhere on $[0,\infty)$. Then \eqref{eq:action_integral:optimal_path_for_H} follows by noting that
	$
	\mathcal{H}(\nabla f(y))\geq \nabla f(y)\cdot v-\mathcal{L}(v)
	$
	for all $y \in \mathbb{T}^d$ and $v \in \mathbb{R}^d$, by convex duality. In particular, 
	$
	\mathcal{H}(\nabla f(x(s)))\geq \nabla f(x(s))\cdot \dot{x}(s)-\mathcal{L}(\dot{x}(s)),
	$
	and integrating gives one inequality in \eqref{eq:action_integral:optimal_path_for_H}. Regarding the other inequality, since $\dot{x}\in\partial \mathcal{H}(\nabla f(x))$, we know that for almost every $t \in [0,\infty)$ and for all $p\in\mathbb{R}^d$, we have
	$
	\mathcal{H}(p)
	\geq 
	\mathcal{H}(\nabla f(x(t)))+\dot{x}(t)\cdot(p-\nabla f(x(t))).
	$
	Therefore, a.e. on $[0,\infty)$,
	\begin{align*}
	\mathcal{H}(\nabla f(x(t))) &\leq \nabla f(x(t))\cdot\dot{x}(t)-\sup_{p\in\mathbb{R}^d}\left[p\cdot \dot{x}(t)-\mathcal{H}(p)\right] \\&= \nabla f(x(t))\cdot\dot{x}(t)-\mathcal{L}(\dot{x}(t)),
	\end{align*}
	and integrating gives the other inequality.
	\medskip
	
	For solving the subdifferential equation, define $F: \mathbb{R}^d \to 2^{ \mathbb{R}^d}$ by $F(x):=\partial \mathcal{H}(\nabla f(x))$, where the function $f \in C^1(\mathbb{T}^d)$ is regarded as a periodic function on $\mathbb{R}^d$. We apply Lemma 5.1 in \cite{De92} for solving $\dot{x}\in F(x)$. The conditions of Lemma 5.1 in the case of $\mathbb{R}^d$ are satisfied if the following holds: $\sup_{x \in \mathbb{R}^d} \|F(x)\|_{\text{sup}}$ is finite, for all $x\in \mathbb{R}^d$, the set $F(x)$ is non-empty, closed and convex, and the map $x\mapsto F(x)$ is upper semicontinuous.
	\medskip
	
	For $\xi \in F(x)$, note that for all $p\in\mathbb{R}^d$ 
	$
	\xi \cdot (p-\nabla f(x))\leq \mathcal{H}(p)-\mathcal{H}(\nabla f(x))
	$.
	Therefore, by shifting $p=p^\prime+\nabla f(x)$, we obtain for all $p^\prime\in\mathbb{R}^d$ that
	$
	\xi \cdot p^\prime\leq \mathcal{H}(p^\prime+\nabla f(x))-\mathcal{H}(\nabla f(x))
	$.
	By continuous differentiability and periodicity of~$f$, and continuity of~$\mathcal{H}$, the right-hand side is bounded in $x$, and we obtain
	\begin{align*}
	\sup_{x \in \mathbb{R}^d} \sup_{\xi \in F(x)}|\xi|
	&=
	\sup_{x \in \mathbb{R}^d} \sup_{\xi \in F(x)} \sup_{|p^\prime| = 1} \xi \cdot p^\prime\\
	&\leq 
	\sup_{x \in \mathbb{R}^d} \sup_{\xi \in F(x)} \sup_{|p^\prime|=1}\left[\mathcal{H}(p^\prime+\nabla f(x))-\mathcal{H}(\nabla f(x))\right] < \infty.
	\end{align*}
	For any $x \in \mathbb{R}^d$, the set $F(x)$ is non-empty, since the subdifferential of a proper convex function $\mathcal{H}(\cdot)$ is nonempty at points where $\mathcal{H}(\cdot)$ is finite and continuous (see e.g.~\cite[Th.~23.4]{rockafellar1966characterization}). Furthermore, $F(x)$ is convex and closed, which follows from the properties of a subdifferential set.
	\medskip
	
	Regarding upper semicontinuity, recall the definition from \cite{De92}: the map $F:\mathbb{R}^d \to 2^{\mathbb{R}^d} \setminus \{\emptyset\}$ is upper semicontinuous if for all closed sets $A\subseteq\mathbb{R}^d$, the set $F^{-1}(A)\subseteq \mathbb{R}^d$ is closed, where
	$
	F^{-1}(A)=\{x\in \mathbb{R}^d\;|\;F(x)\cap A\neq\emptyset\}.
	$
	Let $A\subseteq\mathbb{R}^d$ be closed and $x_n\rightarrow x$ in $\mathbb{R}^d$, with $x_n\in F^{-1}(A)$. That means for all $n\in\mathbb{N}$ that the sets
	$
	\partial \mathcal{H}(\nabla f(x_n))\cap A
	$
	are non-empty, and consequently, there exists a sequence $\xi_n \in F(x_n)\cap A$. We proved above that the set $F(y)\cap A$ is uniformly bounded in $y \in \mathbb{R}^d$. Hence the sequence~$\xi_n$ is bounded, and passing to a subsequence if necessary, it converges to some~$\xi$.
	By definition of $F(x_n)$, for all $p\in\mathbb{R}^d$,
	\begin{align*}
	\xi_n(p-\nabla f(x_n))
	\leq 
	\mathcal{H}(p)-\mathcal{H}(\nabla f(x_n)).
	\end{align*}
	Passing to the limit, we obtain that for all $p \in \mathbb{R}^d$,
	$$
	\xi (p-\nabla f(x))
	\leq 
	\mathcal{H}(p)-\mathcal{H}(\nabla f(x)).
	$$
	This implies by definition that $\xi\in\partial \mathcal{H}(\nabla f(x))$. Since $\xi_n\in A$ and $A$ is closed, we have $\xi\in A$. Hence $x\in F^{-1}(A)$, and $F^{-1}(A)$ is indeed closed.
\end{proof}
\begin{proof}[Proof of (ii) in Proposition \ref{prop:action_integral:bold_H}]
	The comparison principle for the operator $\mathbf{H}$ follows from the fact that~$\mathbf{H}f=\mathcal{H}(\nabla f)$ depends on~$x$ only via gradients. Indeed, for subsolutions $u_1$ and supersolutions $u_2$ of $(1-\tau \mathbf{H})u =h $, we have
	$
	\max(u_1-u_2)
	\leq 
	\tau [\mathcal{H}
	\left(
	\nabla f_1(x_\delta)
	\right)
	-
	\mathcal{H}\left(\nabla f_2(x_\delta^\prime)\right)]
	+
	h(x_\delta)-h(x_\delta^\prime),
	$
	with test functions $f_1,f_2\in\mathcal{D}(H)$ satisfying
	$
	\nabla f_1(x_\delta)
	=
	\nabla f_2(x_\delta^\prime),
	$
	and~$\text{dist}(x_\delta,x_\delta^\prime)\rightarrow 0$ as~$\delta \to 0$. Therefore $\mathcal{H}\left(\nabla f_1(x_\delta)\right)-\mathcal{H}\left(\nabla f_2(x_\delta^\prime)\right) = 0$, and~$\max(u_1-u_2)\leq 0$ follows by taking the limit $\delta\rightarrow 0$.
\end{proof} 
\begin{proof}[Proof of (iii) in Proposition \ref{prop:action_integral:bold_H}]
	Let $u \in C(\mathbb{T}^d)$ be a viscosity solution of the equation~$(1 - \tau H) u = h$. By Lemmas~\ref{lemma:LDP_switching_MP:local_op_strong_sol} and~\ref{lemma:LDP_switching_MP:H_to_H1-H2}, $u$ is a strong viscosity subsolution of $(1 - \tau H_1) u = h$ and a strong viscosity supersolution of $(1- \tau H_2) u = h$. In the proof of Lemma~\ref{lemma:LDP_switching_MP:H1_leq_H2} we obtained $\mathcal{H}_1 \leq \mathcal{H} \leq \mathcal{H}_2$, which in particular implies the inequalities 
	$
	- H_1 \geq -\mathbf{H} \geq -H_2.
	$
	With that, we find that $u$ is both a strong viscosity sub- and supersolution of $(1 - \tau \mathbf{H}) u = h$.
\end{proof}
\section{Proof of large deviations for molecular motors}
In this section, we consider the stochastic process~$(X^n, I^n)$ of Defintion~\ref{def:intro:cont_MM} and prove Theorems~\ref{thm:results:LDP_cont_MM_Limit_I} and~\ref{thm:results:LDP_cont_MM_Limit_II}. The generator~$L_n$ of~$(X^n, I^n)$ is given by
\begin{multline*}
L_n f(x,i)= \frac{1}{n} \frac{1}{2} \Delta_x f(\cdot,i) (x) +b^i(nx) \cdot \nabla_x f(\cdot,i) (x)\\
+ \sum_{j = 1}^J \gamma(n) r_{ij}(nx)
\left[
f(x,j) - f(x,i)
\right],
\end{multline*}
with state space $E_n = \mathbb{T}^d \times \{1,\dots, J\} = \{(x,i)\}$, drifts $b^i \in C^\infty(\mathbb{T}^d)$, jump rates $r_{ij} \in C^\infty(\mathbb{T}^d;[0,\infty))$, and~$\gamma(n)>0$. We frequently write $f(x,i) = f^i(x)$. The nonlinear generators defined by~$H_n f = \frac{1}{n} e^{-nf} L_n e^{nf(\cdot)}$ are given by
\begin{multline}\label{eq:LDP_MM:cont:H_varepsilon}
H_n f(x,i)= \frac{1}{n}\frac{1}{2} \Delta_x f^i(x)+\frac{1}{2}|\nabla_x f^i(x)|^2+b^i\left(nx\right)\nabla_x f^i(x) 
\\+ \frac{1}{n}\gamma(n) \sum_{j=1}^J r_{ij}\left(nx\right)
\left[
e^{ n\left(f(x,j)-f(x,i)\right)}-1
\right].
\end{multline}
\subsection{Proof of Theorem \ref{thm:results:LDP_cont_MM_Limit_I}}
\label{subsubsection:LDP_cont_MM_Limit_I}
\begin{proof}[Verification of~\ref{MM:item:T1} of Theorem \ref{thm:results:LDP_switching_MP}]
    Recall that~$\gamma(n)=n$.
	Choosing the functions~$f_n(x,i)=f(x)+\frac{1}{n}\,\varphi\left(nx, i\right))$, we find
	\begin{multline*}
	H_n f_n(x,i)
	=
	\frac{1}{n}\frac{1}{2}\Delta f(x)
	+
	\frac{1}{2}\Delta_y\varphi^i\left(nx\right)
	+
	\frac{1}{2}\big|\nabla f(x)+\nabla_y\varphi^i\left(nx\right)\big|^2
	\\+
	b^i\left(nx\right)\left(\nabla f(x)+\nabla_y\varphi^i\left(nx\right)\right)	\\
	+
	\sum_{j = 1}^J r_{ij}\left(nx\right)
	\left[
	e^{\varphi\left(nx, j\right)-\varphi\left(nx, i\right)}-1
	\right],
	\end{multline*}
	where~$\nabla_y$ and~$\Delta_y$ denote the gradient and Laplacian with respect to the variable~$y = nx$. The only term of order~$\frac{1}{n}$ that remains is~$\frac{1}{n}\,\Delta f(x)/2$. This suggests to take the remainder terms as the definition of the multivalued operator $H$. In the notation of Theorem \ref{thm:results:LDP_switching_MP}, we choose $E^\prime = \mathbb{T}^d \times\{1,\dots,J\}$ as the state space of the macroscopic variables, and define
	\begin{align}\label{MM:eq:limit-H-proof-cont-MM-model}
	H:=
	\left\{
	(f, H_{f,\varphi}) \, : \,
	f \in C^2(\mathbb{T}^d), \;
	H_{f, \varphi} \in C(\mathbb{T}^d \times E^\prime) \text{ and }
	\varphi \in C^2(E^\prime)
	\right\},
	\end{align}
	with the image functions~$H_{f,\varphi} : \mathbb{T}^d \times E^\prime \to \mathbb{R}$ defined by
	\begin{multline}
	H_{f,\varphi}(x,y,i)
	:=
	\frac{1}{2}\Delta_y\varphi^i(y) + \frac{1}{2} \big|\nabla f(x)+\nabla_y\varphi^i(y) \big|^2 +b^i(y)(\nabla f(x) + \nabla_y\varphi^i(y))
	\\+
	\sum_{j = 1}^J r_{ij}(y)\left[e^{\varphi(y, j)-\varphi(y, i)}-1\right],
	\label{eq:LDP_MM:contI:limit_op_H}
	\end{multline}
	where we write $\varphi = (\varphi^1, \dots, \varphi^J)$ via the identification $C^2(E^\prime) \simeq (C^2(\mathbb{T}^d))^J$.
	\medskip
	
	We now verify~\ref{MM:item:C1},~\ref{MM:item:C2} and~\ref{MM:item:C3} of~\ref{MM:item:T1}. For~\ref{MM:item:C1}, define the maps $\eta_n^\prime:E_n \to E^\prime$ by $\eta_n^\prime(x,i) := (nx,i)$, and recall that the maps~$\eta_n:E_n\to \mathbb{T}^d$ are the projections~$\eta_n(x,i) := x$. For any $(x,y,i) \in \mathbb{T}^d \times E^\prime$, we search for elements $(y_n,i_n) \in \mathbb{T}^d \times \{1,\dots,J\}$ such that both $\eta_n(y_n,i_n) \to x$ and $\eta_n^\prime(y_n,i_n) \to (y,i)$ as $n \to \infty$. 
	For $d=1$, the point $y_n := \frac{1}{n}(\lfloor nx \rfloor + y)$ satisfies $y_n \to x$ and $n y_n = y$ in $\mathbb{T}^d$ (i.e.\ modulo $1$). For $\geq2$ this construction can be done for each coordinate. Therefore,~\ref{MM:item:C1} holds with $y_n = \frac{1}{n}(\lfloor nx \rfloor + y)$ and $i_n = i$.
	\medskip
	
	Regarding Item~\ref{MM:item:C2}, let $(f, H_{f,\varphi}) \in H$. The function~$f_n$ defined by $f_n(x,i):=f(x)+\frac{1}{n}\,\varphi\left(nx, i\right)$ satisfies
	$$
	\|f\circ\eta_n-f_n\|_{L^\infty(E_n)}
	=
	\sup_{(x,i)\in E_n}|f(x)-f_n(x,i)|
	=
	\frac{1}{n}\cdot \|\varphi \|_{L^\infty(E_n)}\xrightarrow{n\rightarrow \infty}0,
	$$
	and
	\begin{align*}
	\|H_{f,\varphi}\circ\eta_n^\prime-H_n f_n\|_{L^\infty(E_n)}
	&=
	\sup_{(x,i)\in E_n}|H_{f,\varphi}(x,nx,i)-H_n f_n(x,i)|\\
	&= 
	\frac{1}{n}\frac{1}{2}\;\sup_{(x,i)\in E_n}|\;\Delta f(x)|\leq \frac{1}{n} \frac{1}{2}\sup|\Delta f|
	\xrightarrow{n\rightarrow \infty}0.
	\end{align*}
	Item~\ref{MM:item:C3}, the fact that the images $H_{f,\varphi}$ depend on $x$ only via the gradients of~$f$, can be recognized in~\eqref{eq:LDP_MM:contI:limit_op_H}.
\end{proof}
\begin{proof}[Verification of (T2) of Theorem \ref{thm:results:LDP_switching_MP}]
	Let $f$ be a function in $\mathcal{D}(H)=C^2(\mathbb{T}^d)$ and $x\in \mathbb{T}^d$. We establish the existence of a vector function $\varphi=(\varphi^1,\dots,\varphi^J)\in (C^2(\mathbb{T}^d))^J$ such that for all $(y,i) \in E^\prime = \mathbb{T}^d \times \{1,\dots,J\}$ and some constant $\mathcal{H}(\nabla f(x)) \in \mathbb{R}$, we have
	$$
	H_\varphi(\nabla f(x),y,i)
	=
	\mathcal{H}(\nabla f(x)).
	$$
	For the flat torus $E = \mathbb{T}^d$, this means that for fixed $\nabla f(x)=p\in\mathbb{R}^d$, we search for a vector function $\varphi_p$ such that $\tilde{H}_{\varphi_p}(p,y,i) = \mathcal{H}(p)$ becomes independent of the variables $(y,i)\in E^\prime$. We can find this vector function by solving a principal eigenvalue problem. We prove Item~\ref{MM:item:T2} with the following Lemma.
	\begin{lemma}
		Let $E^\prime = \mathbb{T}^d \times \{1, \dots, J\}$ and $H$ be the limit operator~\eqref{MM:eq:limit-H-proof-cont-MM-model}. Then:
		\begin{enumerate}[(a)]
			\item For $f \in \mathcal{D}(H)$, the limiting images $H_{\varphi}(\nabla f(x),y,i)$ are of the form 
			\[
			H_\varphi(\nabla f(x) , y, i)
			=
			e^{-\varphi(y,i)} \left[ (B_p + V_p + R)e^{\varphi} \right] (y,i),
			\]
			with $ p = \nabla f(x) \in\mathbb{R}^d$, and operators $B_p, V_p, R : C^2(E^\prime) \rightarrow C(E^\prime)$ defined as
			\begin{align*}
			(B_p h)(y,i) 
			&:=
			\frac{1}{2} \Delta_y h(y,i) + \left( p +b^i(y)\right)\cdot \nabla_y h(y,i) \\
			(V_p h)(y,i)
			&:=
			\left(\frac{1}{2} p^2 + p\cdot b^i(y)\right) h(y,i), \\
			(R\,h)(y,i)
			&:=
			\sum_{j = 1}^J r_{ij}(y) \left[h(y,j) - h(y,i)\right].
			\end{align*}
			\item 
			For any $p \in \mathbb{R}^d$, there exists an eigenfunction $g_p = (g_p^1,\dots, g_p^J) \in (C^2(\mathbb{T}^d))^J$ with strictly positive component functions, $g^i_p > 0 $ on $\mathbb{T}^d$ for $ i =1,\dots, J$, and an eigenvalue $\mathcal{H}(p) \in \mathbb{R}$ such that
			\begin{equation}
			\left[B_p + V_p + R\right] g_p 
			=
			\mathcal{H}(p)\,g_p.
			\label{eq:LDP_MM:contI:cell_problem}
			\end{equation}
		\end{enumerate}
		\label{lemma:LDP_MM:contI:principal_eigenvalue}
	\end{lemma}
	Now~(T2) follows by (a) and (b), since with $\varphi_p := \log g_p$,
	\begin{align*}
	H_{\varphi_p}(p , y, i)
	&\overset{(a)}{=}
	e^{-\varphi_p(y,i)} \left[ B_p + V_p + R \right] e^{\varphi_p(y,i)}
	\\&=
	\frac{1}{g_p(y,i)} \left[B_p + V_p + R\right] g_p (y,i)
	\overset{(b)}{=}
	\mathcal{H}(p).
	\end{align*}
	\emph{Proof of Lemma \ref{lemma:LDP_MM:contI:principal_eigenvalue}.} Writing $p = \nabla f(x)$, Item~(a) follows directly by regrouping the terms in~\eqref{eq:LDP_MM:contI:limit_op_H}. Regarding Item~(b),~$\left[B_p + V_p +R\right] g_p = \mathcal{H}(p) g_p$ is a system of weakly-coupled nonlinear elliptic PDEs on the flat torus. 
	They are weakly coupled in the sense that the component functions $g_p^i$ are only coupled in the lowest order terms by means of the operator $R$, while the operators $B_p$ and $V_p$ act solely on the diagonal. 
	By Proposition~\ref{proposition:appendix:PrEv:fully_coupled_system}, there exists a $\lambda(p)$ and $g_p > 0$ such that $\left[-B_p -V_p - R \right] g_p = \lambda(p) g_p$. Thereby, $\left[B_p + V_p + R \right] g_p = \mathcal{H}(p) g_p$ follows with the same eigenfunction $g_p > 0$ and the principal eigenvalue $\mathcal{H}(p) = - \lambda(p)$. This finishes the verification of~\ref{MM:item:T2}.
\end{proof}
\begin{proof}[Verification of (T3) of Theorem \ref{thm:results:action_integral_representation}]
	We prove that the principal eigenvalue $\mathcal{H}(p)$ of Lemma \ref{lemma:LDP_MM:contI:principal_eigenvalue} is convex in $p\in\mathbb{R}^d$ and satisfies $\mathcal{H}(0)=0$.
	%
	By Proposition~\ref{proposition:appendix:PrEv:fully_coupled_system}, the eigenvalue 
	$\mathcal{H}(p) = - \lambda(p)$ admits the representation
	\begin{align*}
	\mathcal{H}(p)
	&=
	- \sup_{g > 0} \inf_{z^\prime \in E^\prime} 
	\left\{ 
	\frac{1}{g(z^\prime)}\left[
	(-B_p - V_p - R)g
	\right](z^\prime) 
	\right\} \\
	&=
	\inf_{g > 0} \sup_{z^\prime \in E^\prime}
	\left\{
	\frac{1}{g(z^\prime)}
	\left[
	(B_p + V_p + R)g
	\right](z^\prime)
	\right\} \\
	&=
	\inf_{\varphi} \sup_{z^\prime \in E^\prime}
	\left\{
	e^{-\varphi(z^\prime)}
	\left[(B_p + V_p + R)e^{\varphi}
	\right] (z^\prime)
	\right\}
	=:
	\inf_{\varphi} \sup_{z^\prime \in E^\prime} F(p,\varphi)(z^\prime),
	\end{align*}
	with a map $F$ defined by
	\begin{multline*}
	F(p,\varphi)(y,i)
	:=
	\frac{1}{2}\Delta\varphi^i(y)
	+
	\frac{1}{2} |\nabla\varphi^i(y)+p|^2
	+
	b^i(y)(\nabla\varphi^i(y)+p)
	\\+
	\sum_{j = 1}^J r_{ij}(y)
	\left[ e^{\varphi^j(y)-\varphi^i(y)}-1 \right],
	\end{multline*}
	The map~$F$ is jointly convex in $p$ and $\varphi$. For the eigenfunction $\varphi = \varphi_p$, equality holds in the sense that for any $z \in E^\prime$, we have $\mathcal{H}(p) = F(p,\varphi_p)(z)$. Therefore, we obtain for $\tau\in[0,1]$ and any $p_1, p_2 \in \mathbb{R}^d$ with corresponding eigenfunctions $g_1 = e^{\varphi_1}$ and $g_2 = e^{\varphi_2}$ that
	\begin{align*}
	\mathcal{H}(\tau p_1 + (1-\tau) p_2)
	&=
	\inf_\varphi\sup_{E^\prime}F\left(\tau p_1 + (1-\tau) p_2, \varphi\right)\\
	&\leq 
	\sup_{E^\prime}F\left(\tau p_1 + (1-\tau) p_2, \tau \varphi_1+(1-\tau)\varphi_2\right)\\
	&\leq 
	\sup_{E^\prime}\left[\tau F(p_1,\varphi_1) + (1-\tau) F(p_2,\varphi_2)\right]\\
	&\leq 
	\tau \sup_{E^\prime}F(p_1,\varphi_1) + (1-\tau) \sup_{E^\prime} F(p_2,\varphi_2) \\
	&= 
	\tau \mathcal{H}(p_1) + (1-\tau) \mathcal{H}(p_2).
	\end{align*}
	Regarding the claim $\mathcal{H}(0) = 0$, we choose the constant function $\varphi = (1,\dots,1)$ in the variational representation of $\mathcal{H}(p)$. Thereby, we obtain the estimate $\mathcal{H}(0) \leq 0$. For the opposite inequality, we show that for any $\varphi \in C^2(E^\prime)$ 
	$$
	\lambda(\varphi) := \sup_{z^\prime \in E^\prime}
	\left\{
	e^{-\varphi(z^\prime)}
	\left[
	(B_0 + V_0 + R) e^{\varphi}
	\right](z^\prime)
	\right\}
	\geq 0,
	$$
	which then implies $\mathcal{H}(0) = \inf_{\varphi} \lambda(\varphi) \geq 0$. Let $\varphi\in C^2(E^\prime)$; the continuous function $\varphi$ on the compact set $E^\prime$ admits a global minimum $z_m = (y_m, i_m) \in E^\prime$. Thereby, noting that $V_0 \equiv 0$, we find
	\begin{multline*}
	\lambda(\varphi) \geq e^{-\varphi(z_m)}(B_0 + R) e^{\varphi(z_m)}
	=
	\underbrace{\frac{1}{2} \Delta_y \varphi(y_m,i_m)}_{ \displaystyle \geq 0} 
	+
	\frac{1}{2}|\underbrace{\nabla_y \varphi(y_m,i_m)}_{\displaystyle = 0}|^2
	\\+
	b^{i_m}(y_m)\cdot \underbrace{\nabla_y \varphi(y_m,i_m)}_{\displaystyle = 0}
	+
	\sum_{j\neq i} r_{ij}(y_m) 
	\underbrace{\left[ e^{\varphi(y_m,j) - \varphi(y_m,i_m)} - 1 \right]}_{\displaystyle \geq 0}
	\geq
	0.
	\end{multline*}
	%
	This finishes the verification of (T3), and thereby the proof of Theorem \ref{thm:results:LDP_cont_MM_Limit_I}.
\end{proof}
\subsection{Proof of Theorem \ref{thm:results:LDP_cont_MM_Limit_II}}
\label{subsubsection:LDP_cont_MM_Limit_II}	
In this section, we consider the process $(X^n, I^n)$ of Definition~\ref{def:intro:cont_MM} in the limit regime~$\frac{1}{n}\gamma(n) \to \infty$ as~$n \to \infty$. As above in the proof of Theorem \ref{thm:results:LDP_cont_MM_Limit_I}, we start with the nonlinear generator~$H_\varepsilon$ given by~\eqref{eq:LDP_MM:cont:H_varepsilon}, and verify Conditions (T1), (T2) and (T3) of Theorems \ref{thm:results:LDP_switching_MP} and \ref{thm:results:action_integral_representation}.
\begin{proof}[Verification of (T1) of Theorem \ref{thm:results:LDP_switching_MP}]
	We choose functions $f_n(x,i)$ of the form
	$$
	f_n(x,i)
	=
	f(x)
	+
	\frac{1}{n} \, \varphi \left(nx\right)
	+
	\frac{1}{\gamma(n)} \, \xi\left(nx,i
	\right).
	$$
	We abbreviate $y=nx$ in the following equation. Computing~$H_n f_n$ results in
	\begin{multline*}
	H_n f_n (x,i)
	=
	\frac{1}{n}\frac{1}{2}\Delta f(x)
	+
	\frac{1}{2}\left[\Delta\varphi(y)+\frac{n}{\gamma(n)}\Delta\xi^i(y)\right]
	\\+
	\frac{1}{2}
	\big|
	\nabla f(x)+\nabla\varphi(y)+\frac{n}{\gamma(n)}\nabla\xi^i(y)
	\big|^2
	+
	b^i(y)\left(\nabla f(x)+\nabla\varphi(y)
	+
	\frac{n}{\gamma(n)} \nabla\xi^i(y)\right) 
	\\+
	\frac{1}{n}\gamma(n)\sum_{j = 1}^J r_{ij}(y)\left[e^{n(\xi(y,j)-\xi(y,i))/\gamma(n)}-1\right].
	\end{multline*}
	The~$n/\gamma(n)$ terms vanish as~$n\rightarrow\infty$. The last term satisfies
	$$
	\frac{1}{n}\gamma \sum_{j = 1}^J r_{ij}(y) \left[
	e^{n(\xi^j-\xi^i)/\gamma}-1
	\right]
	=
	\sum_{j = 1}^Jr_{ij}(y) \left[
	\xi^j(y)-\xi^i(y)
	\right]
	+
	o_{n\to\infty}(1).
	$$
	Therefore, we choose again $E^\prime := \mathbb{T}^d \times \{1,\dots, J\}$ as the state space of the macroscopic variables, and use the following limit operator $H$,
	\begin{align}\label{MM:eq:multi-H-proof-cont-averaged}
	H:=
	\left\{
	(f, H_{f,\varphi,\xi} \, : \,
	f \in C^2(\mathbb{T}^d) \text{ and } 
	H_{f, \varphi,\xi} \in C(\mathbb{T}^d \times E^\prime)
	\right\},
	\end{align}
	with functions $\varphi$ and $\xi$ in the sets $\varphi \in C^2(\mathbb{T}^d)$ and $\xi = (\xi^1, \dots, \xi^J) \in C^2(E^\prime) \simeq (C^2(\mathbb{T}^d))^J$. The image functions $H_{f,\varphi,\xi} : \mathbb{T}^d \times \mathbb{T}^d \times \{1,\dots,J\} \to \mathbb{R}$ are
	\begin{multline}
	H_{f,\varphi,\xi}(x,y,i)
	:=
	\frac{1}{2}\Delta_y\varphi(y)
	+
	\frac{1}{2} |\nabla f(x)+\nabla_y\varphi(y)|^2
	+
	b^i(y)
	\left(\nabla f(x)+\nabla_y\varphi(y)\right) \\
	+
	\sum_{j = 1}^J r_{ij}(y)\left[\xi(y,j)-\xi(y,i)\right].
	\label{eq:LDP_MM:contII:limit_op_H}
	\end{multline}
	Then $H$ satisfies (T1), which is shown by the same line of argument as above in the proof of Theorem \ref{thm:results:LDP_cont_MM_Limit_I}, with the same maps~$\eta_n$ and $\eta_n^\prime$. The image functions depend only on gradients,~$H_{f,\varphi,\xi}(x,y,i)=H_{\varphi,\xi}(\nabla f(x),y,i)$.
\end{proof}
\begin{proof}[Verification of (T2) of Theorem \ref{thm:results:LDP_switching_MP}]
	For any $p \in \mathbb{R}^d$, we establish the existence of functions $\varphi_p \in C^2(\mathbb{T}^d)$ and $\xi \in C^2(E^\prime)$ such that $H_{\varphi, \xi}(p,\cdot)$ becomes constant on $E^\prime = \mathbb{T}^d \times \{1,\dots,J\}$. To that end, we find a constant $\mathcal{H}(p) \in \mathbb{R}$ and $\varphi_p$ and $\xi_p$ such that for all $(y,i) \in E^\prime$, we have
	$$
	H_{\varphi_p,\xi_p}(p,y,i) = \mathcal{H}(p).
	$$
	We reduce the problem to finding a principal eigenvalue.
	\begin{lemma}
		Let $E^\prime = \mathbb{T}^d \times \{1,\dots,J\}$ and let~$H$ be the operator~\eqref{MM:eq:multi-H-proof-cont-averaged}. Then:
		\begin{enumerate}[(a)]
			\item For $f \in \mathcal{D}(H)$, the images $H_{\varphi,\xi}$ are given by
			\[
			H_{\varphi,\xi}(p,y,i)
			=
			e^{-\varphi(y)}\left[ (B^i_{p} + V^i_{p}) e^{\varphi}\right](y)
			+
			\sum_{j = 1}^J r_{ij}(y)\left[\xi(y,j)-\xi(y,i)\right],
			\]
			where $p = \nabla f(x) \in \mathbb{R}^d$, $B_p^i = \frac{1}{2}\Delta + (p + b^i(y))\cdot\nabla$ and multiplication operator~$V_p^i(y) = p^2/2 + p\cdot b^i(y)$.
			\item For any $\varphi$ and $y\in \mathbb{T}^d$, there exists a function $\xi(y,\cdot)$ on $\{1,\dots,J\}$ such that $\xi \in C^2(E^\prime)$ and for all $i=1,\dots,J$,
			\[
			e^{-\varphi}\left[ (B^i_{p} + V^i_{p}) e^{\varphi}\right](y)
			+
			\sum_{j = 1}^J r_{ij}(y)\left[\xi(y,j)-\xi(y,i)\right]
			=
			e^{-\varphi(y)} \left[ B_p + V_p \right] e^{\varphi(y)},
			\]
			where $B_p = \frac{1}{2}\Delta + (p + \overline{b}(y))\cdot\nabla$, $V_p(y) = \frac{p^2}{2} + p\cdot\overline{b}(y)$. In the operators, $\overline{b}(y) := \sum_{i = 1}^J \mu_y(i)b^i(y)$ is the average drift with respect to the stationary measure $\mu_y \in \mathcal{P}(\{1, \dots, J\})$ of the jump process with frozen jump rates~$r_{ij}(y)$.
			\item There exists a strictly positive eigenfunction $g_p$ and an eigenvalue $\mathcal{H}(p) \in \mathbb{R}$ such that
			\begin{equation}
			\left[ B_p + V_p \right] g_p = \mathcal{H}(p) g_p.
			\label{eq:LDP_MM:contII:PrEv_eq}
			\end{equation}
		\end{enumerate}
		\label{lemma:LDP_MM:contII:principal_eigenvalue}
	\end{lemma}
	By (a), (b) and (c), taking $\varphi_p = \log g_p$ and the corresponding $\xi(y,i)$, we obtain~(T2) via
	\begin{align*}
	H_{\varphi_p,\xi}(p,y,i)
	&\overset{(a)}{=}
	e^{-\varphi_p(y)}\left[ B^i_{p} + V^i_{p}\right] e^{\varphi_p(y)}
	+
	\sum_{j\in\mathcal{J}}r_{ij}(y)\left[\xi(y,j)-\xi(y,i)\right]
	\\&\overset{(b)}{=}
	e^{-\varphi_p(y)} \left[ (B_p + V_p)e^{\varphi} \right] (y)
	\overset{(c)}{=}
	\mathcal{H}(p).
	\end{align*}
	\emph{Proof of Lemma \ref{lemma:LDP_MM:contII:principal_eigenvalue}.} Regarding (a), writing $\xi(y,i)=\xi_y(i)$ and $p = \nabla f(x)\in\mathbb{R}^d$, for all $(y,i)\in E^\prime$ we find
	\begin{align*}
	H_{\varphi,\xi}(p,y,i)
	&=
	\underbrace{\frac{1}{2}\Delta\varphi + \frac{1}{2} \big|p + \nabla\varphi \big|^2 + b^i(p+\nabla\varphi)}_{\displaystyle=e^{-\varphi}(B_{p,i}+V_{p,i})e^{\varphi}} 
	+
	\underbrace{\sum_{j = 1}^J r_{ij}(y)
		\left[
		\xi(y,j)-\xi(y,i)
		\right]}_{\displaystyle =:R_y \xi(y,\cdot)(i)},
	\end{align*}			
	with a generator $R_y$ of a jump process with frozen jump rates~$r_{ij}(y)$.
	\medskip
	
	For (b), let $\varphi \in C^2(\mathbb{T}^d)$ and $y\in \mathbb{T}^d$. We wish to find a function~$\xi_y(\cdot) = \xi(y,\cdot) \in C(\{1,\dots,J\})$ such that
	$$
	e^{-\varphi}\left[B_{p,i}+V_{p,i}\right]e^{\varphi} + R_y \xi_y(i)
	$$
	becomes constant in $i = 1,\dots, J$. By the Fredholm alternative, for any vector~$h\in C(\{1,\dots,J\})$, the equation $R_y\xi_y = h$ has a solution $\xi_y(\cdot)\in C(\{1,\dots,J\})$ if and only if $h \perp \text{ker}(R_y^\ast)$. Since $R_y$ is the generator of a jump process on the finite discrete set $\{1, \dots, J\}$ with rates $r_{ij}(y)$, the null space $\text{ker}(R_y^\ast)$ is one-dimensional and spanned by the unique stationary measure $\mu_y\in\mathcal{P}(\{1,\dots,J\})$, which exists by our irreducibility assumption of Theorem~\ref{thm:results:LDP_cont_MM_Limit_II} (e.g.~\cite[Theorem~17.51]{klenke2013probability}).
	Hence
	$
	e^{-\varphi}\left[B_{p,i}+V_{p,i}\right]e^{\varphi} + R_y \xi_y(i)
	=
	h(p,y)
	$
	is independent of $i \in \{1,\dots,J\}$ if and only if
	\[
	\sum_{i = 1}^J \mu_y(i)\left[(h(p,y)-e^{-\varphi}\left[B_{p,i}+V_{p,i}\right]e^{\varphi}\right]=0.
	\]
	This solvability condition leads to
	\begin{align*}
	\sum_{i = 1}^J \mu_y(i)
	\left[(h(p,y)
	-
	e^{-\varphi}
	\left( B_{p,i}+V_{p,i}\right)
	e^{\varphi}
	\right]
	&=
	h(p,y)-
	e^{-\varphi(y)}\left(
	B_p + V_p
	\right)
	e^{\varphi(y)}
	=
	0.
	\end{align*}
	Hence for $h(p,y) := e^{-\varphi(y)}\left[ B_p + V_p \right]e^{\varphi(y)}$, there exists~$\xi(y,i)$ solving the equation~$R_y \xi(y,\cdot) = h$. Furthermore, since the stationary measure is an eigenvector of a one-dimensional eigenspace, and the rates $r_{ij}(\cdot)$ are smooth by assumption, the eigenfunctions $\xi_y$ depend smoothly on $y$ as well, and (b) follows.
	\medskip
	
	For proving (c) in Lemma~\ref{lemma:LDP_MM:contII:principal_eigenvalue}, we note that Equation~\eqref{eq:LDP_MM:contII:PrEv_eq} corresponds to a principal-eigenvalue problem for a second-order uniformly elliptic operator. By Proposition~\ref{proposition:appendix:PrEv:elliptic_op}, the principal eigenvalue problem $\left[ - B_p - V_p \right] g _p = \lambda(p) g_p$ has a solution $g_p > 0$, with eigenvalue $\lambda(p) \in \mathbb{R}$. The same function $g_p$ and the eigenvalue $\mathcal{H}(p) = -\lambda(p)$ solve~\eqref{eq:LDP_MM:contII:PrEv_eq}.
\end{proof}
\begin{proof}[Verification of (T3) of Theorem \ref{thm:results:action_integral_representation}]
	The principal eigenvalue $\mathcal{H}(p)$ is of the form
	$$
	\mathcal{H}(p)
	=
	\inf_\varphi\sup_{y\in \mathbb{T}^d} F\left(p,\varphi\right)(y),
	$$
	with $F$ jointly convex in $p$ and $\varphi$. Convexity of~$\mathcal{H}(p)$ and~$\mathcal{H}(0) = 0$ follow as above in the proof of Theorem~\ref{thm:results:LDP_cont_MM_Limit_I}.
\end{proof}
\section{Proof of symmetry of Hamiltonians}
In Theorem~\ref{thm:results:detailed_balance_limit_I}, we proved that detailed-balance implies symmetric Hamiltonians. The proof was based on a suitable variational representation of the Hamiltonian. In this section, we show in Proposition~\ref{prop:results:detailed_balance_limit_I} how to obtain this representation.
\medskip

Before giving the rigorous proof, we sketch the argument. To that end, we recall the setting. We work with~$E^\prime = \mathbb{T}^d \times \{1, \dots, J\}$ and denote by~$\mathcal{P}(E^\prime)$ the set of probability measures on $E^\prime$. The Hamiltonian~$\mathcal{H}(p)$ is the principal eigenvalue of the cell problem~\eqref{eq:LDP_MM:contI:cell_problem} described in Lemma~\ref{lemma:LDP_MM:contI:principal_eigenvalue}, and satisfies
\begin{equation}
\mathcal{H}(p)
=
\sup_{\mu\in\mathcal{P}(E^\prime)}
\left[
\int_{E^\prime} V_p(z) \,\dd \mu(z) - I_p(\mu)
\right].
\label{eq:results:LDP_MM:DV_var_rep_H(p)}
\end{equation}
In this formula, we have the continuous map~$V_p$ given by
\begin{equation}\label{MM:eq:function-V-in-Hamiltonian}
V_p (x,i) := \frac{1}{2} p^2 - p\cdot \nabla\psi^i(x),
\end{equation}
and the Donsker-Varadhan functional
\begin{equation}
I_p(\mu):=-\inf_{u > 0}\int_{E^\prime} \frac{L_p u}{u}\,\dd \mu,
\label{eq:results:LDP_MM:DV_functional}
\end{equation}
where the infimum is over strictly positive~$u\in C^2(E')$ and the operator~$L_p$ is
\begin{equation}
L_p u(x,i) 
:=
\frac{1}{2} \Delta_x u(x,i) + (p-\nabla \psi_i(x))\cdot \nabla_x u(x,i)
+
\sum_{j = 1}^J r_{ij}(x) \left[u(x,j) - u(x,i)\right].
\label{eq:results:LDP_MM:L_p_in_DV_functional}
\end{equation}
The variational representation~\eqref{eq:results:LDP_MM:DV_var_rep_H(p)} is a special case of Donsker's and Varadhan's representation theorem on principal eigenvalues~\cite{DonskerVaradhan75}.
Under their general conditions, the infimum is taken over functions that are in the domain of the infinitesimal generator of the semigroup generated by~$L_p$. Pinsky showed that the infimum can be taken over~$C^2$ functions if the coefficients appearing in the operator~$L_p$ are sufficiently regular (Theorem~1.4 in~\cite{pinsky1985evaluating}, Equation~(3.1) in~\cite{pinsky2007regularity}).
\medskip

Since it is not clear from~\eqref{eq:results:LDP_MM:DV_var_rep_H(p)} that~$\mathcal{H}(p)$ is symmetric under the detailed-balance condition, we shall perform a suitable shift in the infimum of the functional~\eqref{eq:results:LDP_MM:DV_functional} to obtain a suitable representation. Rewriting in~\eqref{eq:results:LDP_MM:DV_functional} the strictly positive functions as~$u=\exp(\varphi)$, we find
\begin{equation*}
I_p(\mu)=-\inf_{\varphi} \sum_i \int\left[
\frac{1}{2}\Delta\varphi_i + \frac{1}{2}|\nabla\varphi_i|^2
+(p-\nabla\psi_i) \nabla\varphi_i+\sum_j r_{ij}\left(e^{\varphi_j-\varphi_j}-1\right)\right]\dd\mu_i.
\end{equation*}
Suppose that~$\dd\mu_i=\overline{\mu}_i\,\dd x$ with strictly positive~$\overline{\mu}_i$, where~$\dd x$ is the Lebesgue measure on the torus. Then shifting in the infimum as~$\varphi_i\to\varphi_i+\psi_i+\frac{1}{2}\log\overline{\mu}_i$, we find by calculation that
\begin{equation}
\label{MM:eq:DV-functional-after-shift}
I_p(\mu)=\mathcal{R}(\mu)+\int_{E^\prime} V_p \, \dd\mu - K_p(\mu),
\end{equation}
where~$\mathcal{R}(\mu)$ is the Fisher information given by
\begin{equation}
\label{MM:eq:relative-Fisher-information}
\mathcal{R}(\mu) := \frac{1}{8}\sum_i\int_{\mathbb{T}^d}\left|\nabla \left(\log\frac{\overline{\mu}_i}{e^{-2\psi_i}}\right)\right|^2\,\dd\mu_i,
\end{equation}
and~$K_p(\mu)$ is given by
\begin{multline}
K_p(\mu) = \inf_{\phi}\bigg\{\sum_{i = 1}^J \int_{\mathbb{T}^d}\left(
\frac{1}{2}|\nabla \phi_i(x) + p|^2 
-\sum_{j = 1}^J r_{ij}(x)\right) \,\dd \mu_i(x)\\
+\sum_{i, j = 1}^J \int_{\mathbb{T}^d}r_{ij}(x) e^{-2\psi_i(x)}\sqrt{\overline{\mu}_i(x) \overline{\mu}_j(x)}
e^{\psi_j(x) + \psi_i(x)} e^{\phi(x,j) - \phi(x,i)}
\,\dd x\bigg\}.
\label{eq:results:LDP_MM:K_p(mu)_I}
\end{multline}
Plugging formula~\eqref{MM:eq:DV-functional-after-shift} into the variational representation~\eqref{eq:results:LDP_MM:DV_var_rep_H(p)} leads to the desired representation of the Hamiltonian. The transformation we used is equivalent to shifting by~$(1/2)\log(\overline{\mu}_i/\pi_i)$, where~$\pi_i=e^{-2\psi_i}$ is the stationary measure up to multiplicative constant. This transformation is reminiscent of a \emph{symmetrization} discussed in Touchette's notes~\cite[Eq.~(36)]{Touchette2018}. 
Also when formulating the detailed-balance condition with additional constants in~\eqref{MM:eq:detailed-balance}, that is when not shifting the potentials by constants to renormalize, one can include these constants in the shift to arrive at the same conclusions.
\medskip

In order to make the strategy as outlined above rigorous, we prove that we can restrict to measures~$\mu$ having the required regularity properties. The central step is to exploit the fact that~$I_p(\mu)$ is finite since~$\mathcal{H}(p)$ is finite. By a result of Stroock~\cite[Theorem 7.44]{stroock2012introduction}, finiteness of the Donsker-Varadhan functional implies certain regularity properties in case the generator is reversible. Since the generator~$L_p$ is not reversible, we further bound~$I_p$ by a suitable Donsker-Varadhan functional~$I_\mathrm{rev}$ corresponding to a reversible process in order to be able to apply~\cite[Theorem 7.44]{stroock2012introduction}.
\begin{proposition}
	The Hamiltonian~$\mathcal{H}(p)$ given by~\eqref{eq:results:LDP_MM:DV_var_rep_H(p)} satisfies the following:
	\begin{enumerate}[(a)]
		\item
		The supremum in~\eqref{eq:results:LDP_MM:DV_var_rep_H(p)} can be taken over a smaller set $\mathbf{P}$ of measures, that is
		$$
		\mathcal{H}(p)
		=
		\sup_{\mu\in\mathbf{P}}
		\left[
		\int_{E^\prime} V_p \,\dd\mu - I_p(\mu)
		\right],
		$$
		where
		$ \mathbf{P} \subset \mathcal{P}(E^\prime)$ are the probability measures $\mu = (\mu_1,\dots,\mu_J)$ such that:
		\begin{enumerate}[(P1)]
			\item 
			Each $\mu_i$ is absolutely continuous with respect to the uniform measure on $\mathbb{T}^d$.
			\item 
			For each~$i$, we have~$\nabla (\log \overline{\mu}_i) \in L^2_{\mu_i}(\mathbb{T}^d)$, where $\dd\mu_i(x) = \overline{\mu}_i(x) \dd x$.
		\end{enumerate}
		\item
		We have
		\begin{equation}
		\mathcal{H}(p)
		=
		\sup_{\mu \in \mathbf{P}}
		\left[
		K_p(\mu) - \mathcal{R}(\mu)
		\right],
		\label{eq:results:LDP_MM:new_representation_H(p)}
		\end{equation}
		with the maps~$\mathcal{R}$ and~$K_p$ given by~\eqref{MM:eq:relative-Fisher-information} and~\eqref{eq:results:LDP_MM:K_p(mu)_I} above. In~$K_p(\mu)$, the infimum can be taken over vectors of functions $\phi_i = \phi(\cdot,i)$ such that $\nabla \phi_i \in L^{2}_{\mu_i}(\mathbb{T}^d)$.
		\item
		Under the detailed balance condition,
		\begin{multline}
		K_p(\mu) 
		= 
		\inf_{\phi}\bigg\{
		\sum_{i = 1}^J \int_{\mathbb{T}^d}
		\left(
		\frac{1}{2}|\nabla \phi^i(x) + p|^2 
		-
		\sum_{j = 1}^J r_{ij}(x)
		\right) \,
		\dd \mu_i(x)	\\
		+
		\sum_{i, j = 1}^J \int_{\mathbb{T}^d} r_{ij}(x)e^{-2\psi_i(x)}
		\sqrt{\overline{\mu}_i(x) \overline{\mu}_j(x)}
		e^{\psi_j(x) + \psi_i(x)} \cosh{(\phi(x,j) - \phi(x,i))}
		\,\dd x
		\bigg\}.
		\label{eq:results:LDP_MM:K_p(mu)_I:det_bal}
		\end{multline}
	\end{enumerate}
	\label{prop:results:detailed_balance_limit_I}
\end{proposition}
The representation~\eqref{eq:results:LDP_MM:K_p(mu)_I:det_bal} follows from~\eqref{eq:results:LDP_MM:K_p(mu)_I} by rewriting the sums appearing therein as~$\sum_{ij} a_{ij} = \frac{1}{2} \sum_{ij} (a_{ij} + a_{ji})$, where
$$
a_{ij} = \int_{\mathbb{T}^d}
r_{ij}e^{-2\psi_i}
\sqrt{\overline{\mu}^i(x) \overline{\mu}^j(x)}
e^{\psi^j(x) + \psi^i(x)} e^{\phi(x,j) - \phi(x,i)}
\,\dd x.
$$
This leads to the $\cosh(\cdot)$ terms in \eqref{eq:results:LDP_MM:K_p(mu)_I:det_bal}, and proves (c). We now give the proof of (a) and (b) of Proposition \ref{prop:results:detailed_balance_limit_I}.
\begin{proof}[Proof of (a) in Proposition \ref{prop:results:detailed_balance_limit_I}]
	Let~$p\in\mathbb{R}^d$. The supremum in~\eqref{eq:results:LDP_MM:DV_var_rep_H(p)} can be taken over measures~$\mu$ such that~$I_p(\mu)$ is finite, because~$\mathcal{H}(p)$ is finite and $V_p(\cdot)$ is bounded. We show that finiteness of~$I_p(\mu)$ implies that $\mu$ must satisfy~(P1) and~(P2). To that end, define the map $L_{\text{rev}} : \mathcal{D}(L_{\mathrm{rev}}) \subseteq C(E^\prime) \to C(E^\prime)$ by setting $\mathcal{D}(L_{\mathrm{rev}}) := C^2(E^\prime)$ and
	$$
	L_{\mathrm{rev}} f(x,i)
	:=
	\frac{1}{2} \Delta_x f(x,i) - \nabla \psi_i(x) \cdot \nabla_x f(x,i)
	+
	\overline{\gamma} \sum_{j \neq i} s_{ij}(x)
	\left[
	f(x,j) - f(x,i)
	\right],
	$$
	with jump rates $s_{ij}$ defined as
	$
	s_{ij} \equiv 1
	$
	and
	$
	s_{ji} \equiv e^{2\psi_j - 2\psi_i},
	$
	for $i \leq j$, and with 
	$
	\overline{\gamma} := \sup_{\mathbb{T}^d} \left( r_{ij} / s_{ij}\right) < \infty,
	$
	where $r_{ij}(\cdot)$ are the jump rates appearing in~$L_p$. Furthermore, define $I_{L_{\text{rev}}} : \mathcal{P}(E^\prime) \to [0,\infty]$ by 
	$$
	I_{L_\mathrm{rev}}(\mu)
	:=
	-\inf_{\varphi \in C^2(E^\prime)}
	\int_{E^\prime} e^{-\varphi} L_{\mathrm{rev}} (e^{\varphi}) \, \dd\mu.
	$$
	We shall prove two statements:
	\begin{itemize}
	    \item[(I)] If~$I_{L_{\mathrm{rev}}}(\mu)$ is finite, then the measure~$\mu$ satisfies~(P1) and~(P2).
	    \item[(II)] If~$I_p(\mu)$ is finite, then~$I_{L_{\mathrm{rev}}}(\mu)$ is finite.
	\end{itemize}
	The two statements combined finish the proof.
	\smallskip
	
	Regarding~(I), suppose~$I_{L_\mathrm{rev}}(\mu)$ is finite. Since 
	$
	s_{ij} e^{-2\psi_i}
	=
	s_{ji} e^{-2\psi_j},
	$
	the operator $L_\mathrm{rev}$ admits a reversible measure $\nu_{\text{rev}}$ in $\mathcal{P}(E^\prime)$ given by
	$$
	\nu_{\mathrm{rev}}(A_1, \dots, A_J)
	=
	\frac{1}{\mathcal{Z}} \sum_{i = 1}^J \nu_{\mathrm{rev}}^i(A_i),
	\quad
	\text{ where }
	\dd\nu_{\mathrm{rev}}^i = e^{-2 \psi_i} \dd x
	\text{ and }
	\mathcal{Z} = \sum_i \nu_\mathrm{rev}^i(\mathbb{T}^d).
	$$
	The measure $\nu_{\mathrm{rev}}$ is reversible for $L_{\mathrm{rev}}$ in the sense that for all $f,g \in \mathcal{D}(L_\text{rev})$,
	$$
	\langle L_\mathrm{rev} f, g \rangle_{\nu_\mathrm{rev}}
	=	
	\langle f, L_\mathrm{rev} g \rangle_{\nu_\mathrm{rev}},
	\quad
	\text{ where }
	\langle f, h \rangle_{\nu_\mathrm{rev}}
	=
	\frac{1}{\mathcal{Z}} \sum_i \int_{\mathbb{T}^d} f^i(x)h^i(x) \, \dd\nu_\mathrm{rev}^i(x).
	$$
	By Stroock's result~\cite[Theorem 7.44]{stroock2012introduction},
	\begin{align*}
	I_{L_{\text{rev}}}(\mu)
	&=
	\begin{cases}
	\displaystyle
	-\langle f_\mu, L_\text{rev} f_\mu \rangle_{\nu_\text{rev}} ,& f_\mu = \sqrt{g_\mu} \in D^{1/2} := \mathcal{D}\left( \sqrt{-L_\text{rev}} \right) \text{ and } g_\mu = \frac{\dd\mu}{\dd \nu_\text{rev}}, \\
	\displaystyle +\infty ,& \text{otherwise},
	\end{cases}
	\end{align*}
	where $\dd\mu / \dd\nu_\text{rev}$ is the Radon-Nikodym derivative. In particular, since~$I_{L_\mathrm{rev}}(\mu)$ is finite, we find that~$\mu\ll\nu_{\mathrm{rev}}$ and that~$I_{L_\mathrm{rev}}(\mu)$ is explicitly given by
	\begin{multline}
	I_{L_{\text{rev}}}(\mu)
	=
	-\langle f, L_\text{rev} f \rangle_{\nu_\text{rev}}
	\\=
	\frac{1}{\mathcal{Z}}\sum_{i=1}^J
	\left[
	\int_{\mathbb{T}^d} |\nabla f^i(x)|^2 \, d\nu_\text{rev}^i(x)
	+
	\overline{\gamma} \sum_{j = 1}^J \int_{\mathbb{T}^d} s_{ij}(x)|f^j(x) - f^i(x)|^2 \, d\nu_\text{rev}^i(x)
	\right],
	\label{eq:LDP_MM:det_bal:I_L_rev_explicit}
	\end{multline}
	where we write $f^i = (\dd\mu^i/\dd{\nu^i_{\text{rev}}})^{1/2}$.	Furthermore, $\mu^i$ is absolutely continuous with respect to $\nu^i = e^{-2\psi^i} \dd x$. Since~$e^{-2\psi^i} \dd x \ll \dd x$,
	we find that~$\mu^i$ is absolutely continuous with respect to the volume measure on $\mathbb{T}^d$. Hence (P1) holds true.
	\medskip
	
	We verify (P2) by showing that the integral
	$
	\int_{\mathbb{T}^d} |\nabla (\log \overline{\mu}^i)|^2 \, d\mu^i
	$
	is finite. Let $g^i_\mu := \dd\mu^i / \dd\nu^i_{\mathrm{rev}}$ be the density of $\mu^i$ with respect to $\nu_\mathrm{rev}^i$. Then the densities $\overline{\mu}^i = \dd \mu^i/ \dd x$ satisfy~$g_\mu^i = \overline{\mu}^i e^{2\psi^i}$, because
	$$
	\overline{\mu}^i  
	=
	\frac{\dd\mu^i}{\dd\nu^i_\text{rev}} \frac{\dd\nu_\text{rev}^i}{\dd x}
	=
	\frac{\dd\mu^i}{\dd\nu^i_\text{rev}} e^{-2\psi^i}.
	$$
	Let $f_\mu^i := \sqrt{g_\mu^i}$. By~\eqref{eq:LDP_MM:det_bal:I_L_rev_explicit}, $\int_{\mathbb{T}^d}|\nabla f_\mu^i|^2 \dd\nu_\mathrm{rev}^i$ is finite for every $i = 1,\dots, J$. Hence with the estimate
	\begin{align*}
	\int_{\mathbb{T}^d} |\nabla f_\mu^i|^2 \dd\nu_\text{rev}^i
	&\geq
	\int_{\mathbb{T}^d} |\nabla f_\mu^i|^2 \mathbf{1}_{\{ \overline{\mu}^i > 0\}} \dd\nu_\text{rev}^i
	=
	\frac{1}{4}\int_{\mathbb{T}^d} \frac{|\nabla g_\mu^i|^2}{g_\mu^i} \mathbf{1}_{\{ \overline{\mu}^i > 0\}} \dd\nu^i_\text{rev}	\\
	&=
	\frac{1}{4}\int_{\mathbb{T}^d} \frac{| e^{2\psi^i} \nabla \overline{\mu}^i + 2 \overline{\mu}^i \nabla \psi^i e^{2\psi^i}|^2 }{\overline{\mu}^i} e^{-4\psi^i} \mathbf{1}_{\{ \overline{\mu}^i > 0\}} \dd x\\
	&=
	\frac{1}{4} \int_{\mathbb{T}^d} | \nabla (\log \overline{\mu}^i) + 2 \nabla \psi^i |^2 \mathbf{1}_{\{ \overline{\mu}^i > 0\}} \dd\mu^i	\\
	&\geq
	\frac 18 \int_{\mathbb{T}^d} |\nabla(\log \overline{\mu}^i)|^2 \mathbf{1}_{\{ \overline{\mu}^i > 0\}} \dd\mu^i
	-
	\int_{\mathbb{T}^d} |\nabla \psi^i|^2 \mathbf{1}_{\{ \overline{\mu}^i > 0\}} \dd\mu^i,
	\end{align*}
	we find~$\nabla (\log \overline{\mu}^i) \in L^2_{\mu^i}(\mathbb{T}^d)$.
	\smallskip
	
	Regarding~(II), suppose that~$I_p(\mu)$ is finite. We estimate $r_{ij} / s_{ij}$ from above by $\overline{\gamma} = \sup_{\mathbb{T}^d} (r_{ij} / s_{ij})$ to find
	\begin{multline*}
	I_p(\mu)
	\geq
	\sup_\varphi \sum_i \int_{\mathbb{T}^d}
	-
	\bigg[
	\frac{1}{2} \Delta \varphi^i(x) + \frac{1}{2} |\nabla \varphi^i(x)|^2 + (p - \nabla\psi^i(x)) \nabla \varphi^i(x) 
	\\+ 
	\overline{\gamma} \sum_{j\neq i} s_{ij}(x)
	(e^{\varphi(x,j) - \varphi(x,i)} -1)
	\bigg] \, \dd\mu^i - s_0(\mu),
	\end{multline*}
	where 
	$
	s_0(\mu)
	=
	\sum_{ij} \int_{\mathbb{T}^d} \left[ \overline{\gamma} \, s_{ij}(x) - r_{ij}(x) \right] \dd\mu^i
	$
	is finite. For $p = 0$, this means that $I_0 (\mu) \geq I_{L_\mathrm{rev}}(\mu) - s_0(\mu)$, and hence that $I_{L_\mathrm{rev}}(\mu)$ is finite. 
	\smallskip
	
	For $p \neq 0$, the additional $p$-term can be dealt with by Young's inequality applied as $-p\cdot \nabla \phi^i \geq - p^2/(2 \varepsilon) - \frac{\displaystyle \varepsilon}{2} |\nabla\phi^i|^2$. Thereby,
	\begin{multline*}
	I_p(\mu)
	\geq
	\sup_\varphi \sum_i \int_{\mathbb{T}^d}
	-
	\bigg[
	\frac{1}{2} \Delta \varphi^i(x) + \frac{1 + \varepsilon}{2} |\nabla \varphi^i(x)|^2 + - \nabla\psi^i(x) \nabla \varphi^i(x) 
	\\+ 
	\overline{\gamma} \sum_{j\neq i} s_{ij}(x)
	(e^{\varphi(x,j) - \varphi(x,i)} -1)
	\bigg] \, \dd\mu^i -\frac{p^2}{2\varepsilon}- s_0(\mu)
	\\=
	\frac{1}{\lambda}
	\sup_\varphi \sum_i \int_{\mathbb{T}^d}
	-
	\bigg[
	\frac{1}{2} \Delta \varphi^i(x) + \frac{1}{2} |\nabla \varphi^i(x)|^2 + - \nabla\psi^i(x) \nabla \varphi^i(x) 
	\\+ 
	\lambda \overline{\gamma} \sum_{j\neq i} s_{ij}(x)
	(e^{(\varphi(x,j) - \varphi(x,i)) / \lambda} -1)
	\bigg] \, \dd\mu^i
	-\frac{p^2}{2\varepsilon}- s_0(\mu),		
	\end{multline*}
	where the last equality follows by rescaling $\varphi \rightarrow \varphi / \lambda$, with $\lambda = 1 + \varepsilon > 1$. Therefore, apart from the factor $1/\lambda$ in the exponential term and the multiplicative factor $\lambda \overline{\gamma}$, we obtain the same estimate as above in the $p = 0$ case. Denoting the supremum term in the last line by $I^\lambda_{L_\text{rev}}$, we found the estimate
	\begin{equation}\label{eq:proof-H-formula:estimate-by-I-lambda}
	I_p(\mu)
	\geq
	\frac{1}{\lambda} I^\lambda_{L_\text{rev}}(\mu) - s_p(\mu),
	\end{equation}
	where $s_p(\mu) = (2\varepsilon)^{-1} p^2 + s_0(\mu)$. Hence~$I^\lambda_{L_{\text{rev}}}(\mu)$ is finite. To show that this enforces finiteness of~$I_{L_\text{rev}}(\mu)$, we prove that~$I_{L_\text{rev}}(\mu) = \infty$ implies~$I^\lambda_{L_{\text{rev}}}(\mu) = \infty$.
	\smallskip
	
	If $I_{L_\text{rev}}(\mu) = \infty$, then by definition there exist functions $\varphi_n$ such that
	\begin{multline*}
	a(\varphi_n)
	:=
	-\sum_{i=1}^J \int_{\mathbb{T}^d}
	\bigg[
	\frac{1}{2} \Delta\varphi_n^i + \frac{1}{2} |\nabla \varphi_n^i|^2
	-\nabla\psi^i \nabla \varphi_n^i
	\\+ 
	\overline{\gamma} \sum_{j\neq i} s_{ij}
	\left(
	e^{\varphi_n(x,j) - \varphi_n(x,i)} - 1
	\right)
	\bigg]
	\, \dd\mu^i(x)
	\end{multline*}
	diverges, that is~$a(\varphi_n)\to\infty$ as~$n\to \infty$.
	%
	Write
	\begin{multline*}
	a^\lambda(\varphi_n)
	:=
	-
	\sum_i \int_{\mathbb{T}^d}
	\bigg[
	\frac{1}{2} \Delta \varphi_n^i + \frac{1}{2} |\nabla \varphi_n^i|^2 - \nabla\psi^i \nabla \varphi_n^i 
	\\+ \lambda \overline{\gamma} \sum_{j\neq i} s_{ij}
	(e^{(\varphi_n(x,j) - \varphi_n(x,i)) / \lambda} -1)
	\bigg] \, \dd\mu^i
	\end{multline*}
	for the according evaluation of $\varphi_n$ in $I^\lambda_{L_\text{rev}}(\mu)$. By definition, $I^\lambda_{L_\text{rev}}(\mu) \geq a^\lambda(\varphi_n)$. We show that with~$C=C(\mu)>0$ defined by
	\[
	C(\mu) := \sum_{ij} \int_E \lambda \overline{\gamma} s_{ij}(x) \, \dd\mu^i(x),
	\]
	we have
	\begin{equation}\label{eq:proof:det-bal:estimate-below-with-constant}
	a^\lambda(\varphi_n) \geq a(\varphi_n) - C.
	\end{equation}
	To that end, define the sequences~$\overline{a}_n$ and~$\overline{a}^\lambda_n$ by
	%
	\begin{multline*}
	\overline{a}_n
	:=
	-\sum_{i=1}^J \int_{\mathbb{T}^d}
	\bigg[
	\frac{1}{2} \Delta\varphi_n^i + \frac{1}{2} |\nabla \varphi_n^i|^2
	-\nabla\psi^i \nabla \varphi_n^i
	\\+ 
	\overline{\gamma} \sum_{j\neq i} s_{ij}
	\left(
	e^{\varphi_n(x,j) - \varphi_n(x,i)} \mathbf{1}_{\{\varphi_n(x,j) - \varphi_n(x,i) \geq 0\}} - 1
	\right)
	\bigg]
	\, \dd\mu^i(x),
	\end{multline*}
	and
	%
	%
	\begin{multline*}
	\overline{a}^\lambda_n
	:=
	-
	\sum_i \int_{\mathbb{T}^d}
	\bigg[
	\frac{1}{2} \Delta \varphi_n^i + \frac{1}{2} |\nabla \varphi_n^i|^2 - \nabla\psi^i \nabla \varphi_n^i 
	\\+ 
	\lambda \overline{\gamma} \sum_{j\neq i} s_{ij}
	\left(
	e^{(\varphi_n(x,j) - \varphi_n(x,i)) / \lambda} \mathbf{1}_{\{ \varphi_n(x,j) - \varphi_n(x,i) \geq 0 \}} -1
	\right)
	\bigg] \, \dd\mu^i.
	\end{multline*}
	%
	By the elementary estimates~$e^x \geq e^x \mathbf{1}_{\{x \geq 0\}}$ and~$(e^x\mathbf{1}_{x\geq 0} - e^x) \geq -1$, we obtain the inequalities
	\[
	\overline{a}_n \geq a(\varphi_n) \qquad \text{and} \qquad a^\lambda(\varphi_n) \geq \overline{a}^\lambda_n - C.
	\]
	Furthermore, the bound~$\overline{a}_n\leq \overline{a}^\lambda_n$ is obtained by noting that
	\begin{multline*}
	\overline{a}_n - \overline{a}^\lambda_n
	=
	-\sum_i \int_{\mathbb{T}^d} \overline{\gamma} \sum_{j \neq i} s_{ij}
	\left( e^{\varphi_n^j - \varphi_n^i} \mathbf{1}_{\{ \varphi_n^j - \varphi_n^i \geq 0 \}} - 1 \right) d\mu^i
	\\+
	\lambda \overline{\gamma} \sum_i \int_{\mathbb{T}^d} s_{ij} 
	\left(
	e^{\varphi_n^j - \varphi_n^i} \mathbf{1}_{\{ \varphi_n^j - \varphi_n^i \geq 0 \}} - 1
	\right) \dd\mu^i	\\
	=
	\overline{\gamma} 
	(1 - \lambda)
	\sum_{ij} \int_{\mathbb{T}^d} s_{ij} d\mu^i
	+
	\overline{\gamma} \sum_{ij} \int_{\mathbb{T}^d} s_{ij} 
	\left(
	e^{(\varphi_n^j - \varphi_n^i) / \lambda} - e^{\varphi_n^j - \varphi_n^i}
	\right)
	\mathbf{1}_{\{ \varphi_n^j - \varphi_n^i \geq 0 \}}
	\dd\mu^i,
	\end{multline*}
	which is bounded above by zero since $\lambda = 1 + \varepsilon > 1$ and $e^{x / \lambda} \leq e^{x}$ for $x \geq 0$. In conclusion, we have
	\begin{align*}
	    a^\lambda(\varphi_n) 
	    &\geq \overline{a}^\lambda_n - C \\
	    &\geq \overline{a}_n - C\\
	    &\geq a(\varphi_n) - C,
	\end{align*}
	finishing the proof.
	%
\end{proof}
\begin{proof}[Proof of (b) of Proposition \ref{prop:results:detailed_balance_limit_I}]
	It is sufficient to show that for any $\mu \in \mathbf{P}$, the Donsker-Varadhan functional~$I_p(\mu)$ satisfies~\eqref{MM:eq:DV-functional-after-shift}.
	Integration by parts gives
	\begin{multline*}
	I_p(\mu)
	=
	-\inf_{\varphi} \sum_i \int_{\mathbb{T}^d}	
	\bigg[
	-\frac{1}{2} \nabla\varphi^i \nabla(\log \overline{\mu}^i)
	+ \frac{1}{2}|\nabla\varphi^i|^2
	+(p-\nabla\psi^i) \nabla\varphi^i
	\\+ \sum_j r_{ij}
	\left(
	e^{\varphi^j - \varphi^i} - 1
	\right)
	\bigg]
	\dd \mu^i,
	\end{multline*}
	where $\dd \mu^i = \overline{\mu}^i \dd x$. By a density argument, the infimum can be taken over functions~$\varphi$ such that~$\nabla\varphi_i \in L^{2}_{\mu^i}(\mathbb{T}^d)$. 
	Now shifting in the infimum as~$\varphi_i\to\varphi_i+\frac{1}{2}\log(\overline{\mu}_i) + \psi^i$, we find after some algebra that
	%
	\begin{multline*}
	I_p(\mu) = -\inf_\varphi \sum_i \int_{\mathbb{T}^d}
	\bigg[
	\frac{1}{2} |\nabla\varphi^i + p|^2
	- \frac{1}{2} | (p - \nabla\psi^i) - \frac{1}{2} \nabla \log\overline{\mu}^i|^2
	\\+
	\sum_j r_{ij}
	\left(
	\sqrt{\frac{\overline{\mu}^j}{\overline{\mu}^i}} e^{\psi^j - \psi^i} e^{\varphi^j - \varphi^i}
	- 1\right) \bigg]\,\dd\mu^i.	
	\end{multline*}
	The term containing the square roots and logarithms are not singular since they are integrated against $\dd\mu^i$, so that the integration is over the set $\{\overline{\mu}^i > 0\}$. Now writing out the terms and reorganizing them leads to the claimed equality.
\end{proof}
\appendix
\section{Large-deviation principle implies almost-sure convergence}
It is a well-known fact that a large-deviation principle implies a strong type of convergence of random variables. We provide a sketch of proof here since we know of no reference in the literature. 
Let~$\mathcal{I}:\mathcal{X}\to[0,\infty]$ be a rate function. We denote by~$\{\mathcal{I} = 0 \}$ the set of its global minimizers.
\begin{theorem}
	\label{thm:LDP-implies-as}
	For $n=1,2\dots$, let $X^n$ be a random variable taking values in a Polish space~$(\mathcal{X},d)$. Suppose that~$\{X^n\}_{n \in \mathbb{N}}$ satisfies a large-deviation principle with rate function~$\mathcal{I}$. Then~$d(X^n,\{\mathcal{I} = 0\}) \to 0$ almost surely as~$n \to \infty$.
\end{theorem}
We point out that as specified in Definition~\ref{def:rate-function}, the rate function in Theorem~\ref{thm:LDP-implies-as} is assumed to have compact sub-level sets.
\begin{proof}[Proof of Theorem~\ref{thm:LDP-implies-as}]
For~$k,n\in\mathbb{N}$, let~$A_k^n$ be the event
\begin{equation*}
    A_k^n := \left\{d(X^n,\{\mathcal{I}=0\})\geq 1/k\right\},
\end{equation*}
and write
\begin{equation*}
    A_k^n\;\text{i.o.} := \bigcap_{N\geq 1}\bigcup_{n\geq N} A_k^n.
\end{equation*}
Let~$k\in\mathbb{N}$. By the large-deviation upper bound, there exists a~$\delta>0$ such that for all~$n$ sufficiently large,
    \begin{equation*}
        \mathbb{P}\left(A_k^n\right) \leq e^{-n\delta}.
    \end{equation*}
    Therefore~$\sum_{n=1}^\infty\mathbb{P}\left(A_k^n\right)$ is finite, and by the Borel-Cantelli Lemma,
    \begin{equation*}
        \mathbb{P}\left(A_k^n \; \text{i.o.}\right) = 0.
    \end{equation*}
    With that, almost-sure convergence follows by noting that
    \begin{align*}
        \mathbb{P}\left(\{d(X^n,\{\mathcal{I} = 0\})\xrightarrow{n\to\infty} 0\} \;\text{is not true}\right) \leq \sum_{k=1}^\infty\mathbb{P}\left(A_k^n \; \text{i. o.}\right) = 0.
    \end{align*}
\end{proof}
\section{Principal eigenvalues}
In this section we collect results on principal-eigenvalue problems that we encounter in the proofs of the molecular-models.
\begin{proposition}
	Let~$P$ be a second-order uniformly elliptic operator given by
	\begin{equation}
		P
	= 
		-\sum_{k \ell} a_{k \ell}(\cdot) \frac{\partial^2}{\partial x^k \partial x^\ell} + \sum_k b_k(\cdot) \frac{\partial}{\partial x^k} + c(\cdot),
	\label{eq:appendix:uniform_elliptic_op}
	\end{equation}
	with smooth coefficients $a_{k \ell}, b_k, c \in C^\infty(\mathbb{T}^d)$. Then there exists a strictly positive function~$u \in C^\infty(\mathbb{T}^d)$ and a unique~$\lambda \in \mathbb{R}$ such that~$Pu=\lambda u$, and~$\lambda$ is given by
	$$
		\lambda
	=
		\sup_{g > 0} \inf_{x\in \mathbb{T}^d} \left[\frac{P g(x)}{g(x)}\right]
	=
		\inf_{\mu \in \mathcal{P}(\mathbb{T}^d)} 
		\sup_{g > 0}
	\left[
		\int_{\mathbb{T}^d}
		\frac{Pg}{g} \,d\mu
	\right].
	$$
	\label{proposition:appendix:PrEv:elliptic_op}
	\end{proposition}
	\begin{proposition}
	Let $L : C^2(\mathbb{T}^d)^J \rightarrow C(\mathbb{T}^d)^J$ be a $J\times J $ diagonal matrix of uniformly elliptic operators,
	\begin{align}
	L
	=
		\begin{pmatrix}
    	L^{(1)}	&	&	0	\\
    		& \ddots	&	\\
    	0	&	& L^{(J)}       
  		\end{pmatrix},\;L^{(i)} = -\sum_{k \ell}^J a_{k \ell}^{(i)}(\cdot) \frac{\partial^2}{\partial x^k \partial x^\ell} + \sum_k^J b^{(i)}_k(\cdot)\frac{\partial}{\partial x^k} + c^{(i)}(\cdot),
  	\label{eq:appendix:elliptic_operators_in_coupled_system}
	\end{align}
	with
	$
		a_{k \ell}^{(i)}(\cdot), b_k^{(i)}(\cdot), c^{(i)}(\cdot) \in C^\infty(\mathbb{T}^d),
	$
	and let $R$ be a $J\times J$ matrix with non-negative functions on the off-diagonal,
	\begin{align*}
	R
	=
	\begin{pmatrix}
		R_{11}	&	&	\geq 0	\\
    		& \ddots	&	\\
    	\geq 0	&		&	R_{JJ}
	\end{pmatrix}, 
	\qquad
	R_{ij} \geq 0 \text{  for all  } i \neq j.
	\end{align*}
	Suppose that the matrix $\overline{R}$ with entries $\overline{R}_{ij} := \sup_{y \in \mathbb{T}^d} R_{ij}(y)$ is irreducible.
	Then for the operator $P:= L - R$, there exists a unique~$\lambda \in \mathbb{R}$ and a strictly vector~$u \in \left(C^\infty(\mathbb{T}^d)\right)^J$, $u^i(\cdot) > 0$ for all $i = 1, \dots, J$, such that~$Pu=\lambda u$. Furthermore,~$\lambda$ is given by
	$$
		\lambda
	=
		\sup_{g > 0} \inf_{z \in E^\prime} \left[\frac{P g(z)}{g(z)}\right]
	=
		\inf_{\mu \in \mathcal{P}(E^\prime)} 
		\sup_{g > 0}
	\left[
		\int_{E^\prime}
		\frac{Pg}{g} \,d\mu
	\right].
	$$
	\label{proposition:appendix:PrEv:fully_coupled_system}
	\end{proposition}	
In the above propositions, the eigenvalue~$\lambda$ is referred to as the principal eigenvalue. 
The principal-eigenvalue problem on closed manifolds, such as the torus~$\mathbb{T}^d$, is solved for instance by Padilla~\cite{padilla1997principal}. Donsker and Varadhan's variational representations for principal eigenvalues, \cite{DonskerVaradhan75, DonskerVaradhan76}, apply to the case of compact metric spaces without boundary. A proof of how to obtain the principal eigenvalue for coupled systems of equations is given by Sweers~\cite{Sweers92} and Kifer~\cite{kifer1992principal}. Sweers considers a Dirichlet boundary problem, but his results transfer to the compact setting without boundary. Kifer gives an independent proof for the case of a compact manifold, in Lemma 2.1 and Proposition 2.2 in \cite{kifer1992principal}.
\section*{Acknowledgement}
The authors thank Frank Redig, Francesca Collet and Federico Sau for their remarks and suggestions during a couple
of meetings.
MS also thanks Georg Prokert, Jim Portegies and Richard Kraaij for answering
various questions about principal eigenvalues, measure theory and large deviations. The authors acknowledge
financial support through NWO grant 613.001.552.
\bibliographystyle{alpha}
\bibliography{refmikola}

\begin{thebibliography}{HKM08b}

\bibitem[BD19]{BudhirajaDupuis2019}
A{.} Budhiraja and P{.} Dupuis.
\newblock {\em {A}nalysis and {A}pproximation of {R}are {E}vents:
  {R}epresentations and {W}eak {C}onvergence {M}ethods}, volume~94 of {\em
  Probability Theory and Stochastic Modelling}.
\newblock Springer, 2019.

\bibitem[BDG18]{BudhirajaDupuisGanguly2018}
A{.} Budhiraja, P{.} Dupuis, and A{.} Ganguly.
\newblock Large deviations for small noise diffusions in a fast markovian
  environment.
\newblock {\em Electronic Journal of Probability}, 23, 2018.

\bibitem[Dei92]{De92}
K{.} Deimling.
\newblock {\em Multivalued {D}ifferential {E}quations}, volume~1 of {\em De
  Gruyter Series in Nonlinear Analysis and Applications}.
\newblock Walter de Gruyter \& Co., Berlin, 1992.

\bibitem[DV75]{DonskerVaradhan75}
M{.}~D{.} Donsker and S{.} R{.}~S{.} Varadhan.
\newblock On a {V}ariational {F}ormula for the {P}rincipal {E}igenvalue for
  {O}perators with {M}aximum {P}rinciple.
\newblock {\em Proceedings of the National Academy of Sciences},
  72(3):780--783, 1975.

\bibitem[DV76]{DonskerVaradhan76}
M.~D. Donsker and S.~R.~S. Varadhan.
\newblock {On the principal eigenvalue of second‐order elliptic differential
  operators}.
\newblock {\em Communications on Pure and Applied Mathematics}, 29(6):595--621,
  1976.

\bibitem[EK86]{EthierKurtz1986}
S{.}~N{.} Ethier and T{.}~G{.} Kurtz.
\newblock {\em Markov {P}rocesses: {C}haracterization and {C}onvergence},
  volume 282.
\newblock Wiley, New York, 1986.

\bibitem[Ell99]{Ellis1999}
R{.}~S{.} Ellis.
\newblock The theory of large deviations: from {B}oltzmann's 1877 calculation
  to equilibrium macrostates in 2{D} turbulence.
\newblock {\em Physica D: Nonlinear Phenomena}, 133(1-4):106--136, 1999.

\bibitem[FK06]{FengKurtz2006}
J{.} Feng and T{.}~G{.} Kurtz.
\newblock {\em Large {D}eviations for {S}tochastic {P}rocesses}, volume 131 of
  {\em Mathematical surveys and monographs}.
\newblock American Mathematical Society, 2006.

\bibitem[FL96]{FreidlinLee96}
M{.} Freidlin and T-Y Lee.
\newblock Wave front propagation and large deviations for
  diffusion--transmutation process.
\newblock {\em Probability theory and related fields}, 106(1):39--70, 1996.

\bibitem[FS17]{FaggionatoSilvestri17}
A{.} Faggionato and V{.} Silvestri.
\newblock Random walks on quasi one dimensional lattices: Large deviations and
  fluctuation theorems.
\newblock In {\em Annales de l'Institut Henri Poincar{\'e}, Probabilit{\'e}s et
  Statistiques}, volume~53, pages 46--78, 2017.

\bibitem[HKM08a]{HastingsKinderlehrerMcLeod2008}
S{.} Hastings, D{.} Kinderlehrer, and J{.}~B{.} McLeod.
\newblock Diffusion {M}ediated {T}ransport in {M}ultiple {S}tate {S}ystems.
\newblock {\em SIAM Journal on Mathematical Analysis}, 39(4):1208--1230, 2008.

\bibitem[HKM08b]{HastingsKinderlehrerMcleod08}
S{.} Hastings, D{.} Kinderlehrer, and J{.}~B{.} Mcleod.
\newblock {Diffusion mediated transport with a look at motor proteins}.
\newblock In {\em Recent Advances in Nonlinear Analysis}. World Scientific,
  2008.

\bibitem[HMS16]{HuangMandjesSpreij2016}
G{.} Huang, M{.} Mandjes, and P{.} Spreij.
\newblock Large deviations for {M}arkov-modulated diffusion processes with
  rapid switching.
\newblock {\em Stochastic Processes and their Applications}, 126(6):1785--1818,
  2016.

\bibitem[How01]{Howard2001}
J{.} Howard.
\newblock {\em Mechanics of {M}otor {P}roteins and the {C}ytoskeleton}.
\newblock Sinauer Associates Sunderland, 2001.

\bibitem[HY14]{HeYin2014}
Q{.} He and G{.} Yin.
\newblock Large deviations for multi-scale {M}arkovian switching systems with a
  small diffusion.
\newblock {\em Asymptotic Analysis}, 87(3-4):123--145, 2014.

\bibitem[JAP97]{JulicherAjdariProst1997}
F{.} J{\"u}licher, A{.} Ajdari, and J{.} Prost.
\newblock Modeling {M}olecular {M}otors.
\newblock {\em Reviews of Modern Physics}, 69(4):1269, 1997.

\bibitem[KF07]{KolomeiskyFisher07}
A{.}~B{.} Kolomeisky and M{.}~E{.} Fisher.
\newblock {Molecular Motors: A Theorist's Perspective}.
\newblock {\em Annual Review of Physical Chemistry}, 58(1):675--695, 2007.

\bibitem[Kif92]{kifer1992principal}
Y.~Kifer.
\newblock Principal eigenvalues and equilibrium states corresponding to weakly
  coupled parabolic systems of {PDE}.
\newblock {\em Journal d'Analyse Math{\'e}matique}, 59(1):89--102, 1992.

\bibitem[Kle13]{klenke2013probability}
A.~Klenke.
\newblock {\em Probability Theory: A Comprehensive Course}.
\newblock Springer Science \& Business Media, 2013.

\bibitem[Kol13]{Kolomeisky13}
A{.}~B{.} Kolomeisky.
\newblock Motor proteins and molecular motors: {H}ow to operate machines at the
  nanoscale.
\newblock {\em Journal of Physics: Condensed Matter}, 25(46):463101, 2013.

\bibitem[KP17]{KumarPopovic2017}
R{.} Kumar and L{.} Popovic.
\newblock Large deviations for multi-scale jump-diffusion processes.
\newblock {\em Stochastic Processes and their Applications}, 127(4):1297--1320,
  2017.

\bibitem[Kra16]{Kraaij2016}
R{.}~C{.} Kraaij.
\newblock Large deviations for finite state {M}arkov jump processes with
  mean-field interaction via the comparison principle for an associated
  {H}amilton-{J}acobi equation.
\newblock {\em Journal of Statistical Physics}, 164(2):321--345, 2016.

\bibitem[KS20]{KraaijSchlottke20TR}
R{.}~C{.} Kraaij and M{.}~C{.} Schlottke.
\newblock A large deviation principle for {M}arkovian slow-fast systems.
\newblock {\em arXiv preprint arXiv:2011.05686}, 2020.

\bibitem[MS13]{Mirrahimi2013}
S.~Mirrahimi and P.~E. Souganidis.
\newblock A homogenization approach for the motion of motor proteins.
\newblock {\em Nonlinear Differential Equations and Applications NoDEA},
  20(1):129--147, 2013.

\bibitem[Pad97]{padilla1997principal}
P.~Padilla.
\newblock The principal eigenvalue and maximum principle for second order
  elliptic operators on {R}iemannian manifolds.
\newblock {\em Journal of Mathematical Analysis and Applications},
  205(2):285--312, 1997.

\bibitem[Pin85]{pinsky1985evaluating}
R.~G. Pinsky.
\newblock {On evaluating the Donsker-Varadhan I-function}.
\newblock {\em The Annals of Probability}, pages 342--362, 1985.

\bibitem[Pin07]{pinsky2007regularity}
R.~G. Pinsky.
\newblock {Regularity properties of the Donsker--Varadhan rate functional for
  non-reversible diffusions and random evolutions}.
\newblock {\em Stochastics and Dynamics}, 7(02):123--140, 2007.

\bibitem[PS09a]{PerthameSouganidis09a}
B{.} Perthame and P{.}~E{.} Souganidis.
\newblock Asymmetric potentials and motor effect: A homogenization approach.
\newblock {\em Annales de l'Institut Henri Poincare (C) Non Linear Analysis},
  26(6):2055--2071, 2009.

\bibitem[PS09b]{PerthameSouganidis2009Asymmetric}
B.~Perthame and P.~E. Souganidis.
\newblock Asymmetric potentials and motor effect: A large deviation approach.
\newblock {\em Archive for Rational Mechanics and Analysis}, 193(1):153--169,
  2009.

\bibitem[Roc66]{rockafellar1966characterization}
R{.}~T{.} Rockafellar.
\newblock Characterization of the subdifferentials of convex functions.
\newblock {\em Pacific Journal of Mathematics}, 17(3):497--510, 1966.

\bibitem[Str12]{stroock2012introduction}
D.~W. Stroock.
\newblock {\em An Introduction to the Theory of Large Deviations}.
\newblock Springer Science \& Business Media, 2012.

\bibitem[Swe92]{Sweers92}
G.~Sweers.
\newblock {Strong positivity in {$C(\overline{\Omega})$} for elliptic systems}.
\newblock {\em Mathematische Zeitschrift}, 209(1):251, 1992.

\bibitem[Tou18]{Touchette2018}
H{.} Touchette.
\newblock Introduction to dynamical large deviations of {M}arkov processes.
\newblock {\em Physica A: Statistical Mechanics and its Applications},
  504:5--19, 2018.

\bibitem[Vor11]{Vorotnikov11}
D{.} Vorotnikov.
\newblock The flashing ratchet and unidirectional transport of matter.
\newblock {\em Discrete \& Continuous Dynamical Systems-B}, 16(3):963, 2011.

\bibitem[Vor14]{Vorotnikov14}
D{.} Vorotnikov.
\newblock Analytical aspects of the {B}rownian motor effect in randomly
  flashing ratchets.
\newblock {\em Journal of mathematical biology}, 68(7):1677--1705, 2014.

\bibitem[WPE03]{WangPeskinElston2003}
H{.} Wang, C{.}~S{.} Peskin, and T{.}~C{.} Elston.
\newblock A robust numerical algorithm for studying biomolecular transport
  processes.
\newblock {\em Journal of theoretical biology}, 221(4):491--511, 2003.

\bibitem[YZ10]{yin2010hybrid}
G.~Yin and C.~Zhu.
\newblock {\em Hybrid Switching Diffusions: Properties and Applications}.
\newblock Springer New York, 2010.

\end{thebibliography}
\end{document}